\documentclass[10pt, A4paper]{amsart}
\usepackage{hyperref}

\usepackage{amssymb,amsmath,mathtools}
\usepackage{enumerate}
\usepackage{color}
\usepackage{mathrsfs}
\usepackage{enumitem}
\usepackage{graphicx}
\usepackage{float}
\usepackage{tikz}
\usepackage{mathtools}

\definecolor{colorHyperLinks}{HTML}{000000}

\hypersetup{
colorlinks,
allcolors=colorHyperLinks,
allbordercolors=.
}

\usepackage{setspace}

\usepackage[a4paper, left=31mm, right=31mm, top=30mm, bottom=25mm]{geometry}

\newtheorem{theorem}{Theorem}[section]
\newtheorem{proposition}[theorem]{Proposition}
\newtheorem{lemma}[theorem]{Lemma}
\newtheorem{corollary}[theorem]{Corollary}
\theoremstyle{definition}
\newtheorem{definition}[theorem]{Definition}
\newtheorem{example}[theorem]{Example}
\newtheorem{remark}[theorem]{Remark}

\newcommand{\indep}{\perp \!\!\! \perp}
\newcommand{\marriage}{\tikz[baseline=-0.25em, line width=.04em]{%
\draw (0,0) circle[radius=.20em];
\draw (0.20em,0) circle[radius=.20em];
}}

\DeclareFontFamily{U}{mathx}{}
\DeclareFontShape{U}{mathx}{m}{n}{ <-> mathx10 }{}
\DeclareSymbolFont{mathx}{U}{mathx}{m}{n}
\DeclareFontSubstitution{U}{mathx}{m}{n}

\DeclareMathAccent{\widecheck}{0}{mathx}{"71}

\makeatletter
\newcommand{\wcheck}[1]{\mathpalette\wcheck@{#1}}
\newcommand{\wcheck@}[2]{%
\begingroup
\edef\wcheck@font{\the
\ifx#1\displaystyle\textfont\else\ifx#1\textstyle\textfont
\else\ifx#1\scriptstyle\scriptfont\else\scriptscriptfont\fi\fi\fi\@ne
}%
\sbox\z@{\wcheck@font\mbox{#2}\mbox{\char\the\skewchar\font}}%
\sbox\tw@{\wcheck@font#2\char\the\skewchar\font}%
\dimen@=\dimexpr\wd\tw@-\wd\z@\relax
{\,\kern2\dimen@\widecheck{\!\kern-2\dimen@#2\!}\,}%
\endgroup
}
\makeatother

\newcommand{\comment}[1]{}
\newcommand{\F}{\mathcal F}
\newcommand{\G}{\mathcal G}
\renewcommand{\H}{\mathcal H}

\newcommand{\Hs}{\mathcal H}
\newcommand{\AW}{\mathcal A\mathcal W}
\newcommand{\Law}{\mathscr L}
\newcommand{\law}{\Law}

\newcommand{\ip}{\mathrm{ip}}
\newcommand{\id}{\mathrm{id}}

\newcommand{\pj}{\text{pj}}
\newcommand{\simhk}{\sim_{\mathrm{HK}}}

\newcommand{\cS}{\mathcal{S}}

\newcommand{\fp}[1]{{\mathbb #1}}

\newcommand{\sawu}[1]{{\widecheck{#1}}}
\newcommand{\fpsawu}[1]{{\widecheck{\mathbb #1}}}
\newcommand{\iuc}[1]{{#1^\sharp}}

\newcommand{\W}{\mathcal W}
\newcommand{\N}{\mathbb N}
\newcommand{\R}{\mathbb R}
\newcommand{\X}{\mathcal X}
\newcommand{\Y}{\mathcal Y}
\newcommand{\Z}{\mathcal Z}
\newcommand{\A}{\mathcal A}
\newcommand{\B}{\mathcal B}
\newcommand{\Pc}{\mathcal P}

\newcommand{\cpl}{\mathrm{CPL}}
\newcommand{\cplc}{\cpl_{\mathrm{c}}}
\newcommand{\cplbc}{\cpl_{\mathrm{bc}}}

\newcommand{\scpl}{\mathrm{cpl}}

\newcommand{\scplbc}{\scpl_{\mathrm{bc}}}

\newcommand{\fpcpl}{\mathrm{PrCpl}}

\def\P{{\mathbb P}}
\def\E{{\mathbb E}}

\def\PP{{\mathcal P}}
\def\Q{{\mathbb Q}}
\def\R{{\mathbb R}}

\newcommand{\FP}{\mathcal{FP}}
\newcommand{\FFP}{\mathrm{FP}}

\usepackage{relsize}
\def\fcmp{\mathbin{\raise 0.6ex\hbox{\oalign{\hfil$\scriptscriptstyle \mathrm{o}$\hfil\cr\hfil$\scriptscriptstyle\mathrm{9}$\hfil}}}}

\numberwithin{equation}{section}

\author{Mathias Beiglb\"ock, Susanne Pfl\"ugl, \and Stefan Schrott}
\thanks{The authors acknowledge FWF-support through projects Y782 and P35197. We also thank Gudmund Pammer for many helpful discussions and in particular the construction in Example \ref{ex:amalgamation}.}
\date{\today}

\begin{document}
\title{A probabilistic view on the Adapted Wasserstein Distance}

\begin{abstract}
Causal optimal transport and adapted Wasserstein distance have applications in different fields from optimization to mathematical finance and machine learning. The goal of this article is to provide equivalent formulations of these concepts in classic probabilistic language. In particular, we prove a Skorokhod representation theorem for adapted weak convergence,  reformulate the equivalence of stochastic processes using Markovian lifts,  and  give an expression for the adapted Wasserstein distance based on representing processes on a common stochastic basis. 

\end{abstract}

\maketitle

\section{Introduction}%
Authors from several fields have independently introduced \emph{extended} weak topologies on the class of stochastic processes which refine the usual weak topology and take the temporal flow of information into account, see \cite{Al81, He96, HeSc02, BiTa19, PfPi12,PfPi14, NiSu20} among others. 
We will consider them 
mainly from the perspective of the \emph{adapted Wasserstein distance} which metrizes the adapted weak topology. Moreover, it allows for quantitative estimates (e.g.\ Lipschitz continuity of optimal control problems \cite{AcBaZa20} and pricing/hedging problems in mathematical finance \cite{Do14, GlPfPi17, BaBaBeEd19a}) and makes it possible to extend geometric concepts like McCann's interpolation and Wasserstein barycenters to stochastic processes, see \cite{BaBePa21}. 

The approaches mentioned above have in common that the original formulations appear somewhat technical and use concepts that are not widely used in stochastics. The goal of this paper is to provide reformulations in classical probabilistic terms.
Specifically, we reformulate equivalence of stochastic processes in the sense of Hoover--Keisler \cite{HoKe84} using Markov lifts, we establish a Skorokhod representation theorem for adapted weak convergence and a representation of the adapted Wasserstein distance by equivalent stochastic processes that are defined on a common stochastic basis. In order to make the treatise as accessible as possible, we restrict to the case of finite discrete time. 

\subsection{Equivalence of stochastic processes}

The main objects we consider are filtered processes. 
\begin{definition}\label{FPDef}
A five-tuple
\begin{align}\label{eq:FP} \fp{X}:= (\Omega, \F, \P, (\F_t)_{t=1}^N, (X_t)_{t=1}^N),\end{align} 
where $X=(X_t)_{t=1}^N$ is adapted to $(\F_t)_{t=1}^N$ and takes values\footnote{We restrict to $\R^d$ in the introduction for ease of notation. However, our results are equally true for abstract Polish spaces and in fact this setting will be technically more convenient for some of the later arguments.} in $\R^d$ 
is called a \emph{filtered process}. We  call $(\Omega, \F, \P, (\F_t)^N_{t=1})$ the \emph{stochastic basis} of $\fp X$.

The class of all filtered processes is denoted by $\FP$, the subclass of all filtered processes with finite $p$-th moment, i.e.\ $\E[ |X|^p]<\infty$ is denoted by $\FP_p$. 
\end{definition}
We emphasize the importance of using general filtrations rather than just the one generated by $X$ itself: On the one hand the filtration can encode information which is crucial to understand the nature of the process, this is argued e.g.\ by Aldous in \cite[Section 11]{Al81}. To give simple examples, when processes have the same law but different filtration, the values of optimal stopping or stochastic control problems may vary, their Doob-decomposition may be different, etc. On the other hand, it is also convenient for practical reasons: \emph{only} when allowing for general filtrations, the adapted Wasserstein distance turns out to be complete and geodesic. We come back to this point below.


\medskip

Allowing for arbitrary stochastic bases (as in Definition \ref{FPDef}) obviously generates redundancy. Roughly speaking, we want to identify $\fp{X}, \fp {Y}$ if they are `probabilistically equivalent'. 
 In particular, this should imply that they have the same laws but it should also entail properties that depend on the filtration, e.g.\ being Markov or the outcome of stochastic optimization problems. 
Based on considering processes together with their associated prediction processes, Aldous \cite{Al81} introduced the notion of synonymity of processes which is stronger than having the same law. However, synonymous processes may still exhibit different properties,  e.g.\ there are synonymous processes that yield different values in optimal stopping problems, see \cite[Chapter~6]{BaBePa21}. 
Hoover--Keisler \cite{HoKe84} gave a stricter notion of equivalence of processes based on an infinite iteration of the Aldous' prediction process construction and showed that this notion of equivalence ensures that the processes have the same probabilistic properties. As it is technically quite involved to make this precise, we refer to \cite{HoKe84} for the precise definition and give an equivalent characterization below. We will denote the Hoover--Keisler equivalence as $\fp X \simhk \fp Y$. Identifying $\simhk$-equivalent processes, yields the factor space 
\[
\FFP:= \FP /_{\simhk}.
\]

A contribution of the present article is to provide a more accessible formulation for the equivalence of processes in the sense of Hoover--Keisler based on Markovian lifts. 
A filtered process $\fp{X}$ is \emph{Markovian} if $\Law(X_{t+1} |\F_t)=\Law(X_{t+1} |X_t)$ for all $t < N$. 
Filtered processes can be lifted to Markov processes by adding a further `information coordinate' $\iuc{X}$, which `stores'  information inherent in the filtration of the process.

\begin{definition}
\label{def:Marlov-lift_Intro}
Let $\fp{X}=(\Omega, \F, \P, (\F_t)_{t=1}^N, (X_t)_{t=1}^N)$ be a filtered process and $\iuc{X}= (\iuc{X_t})_{t=1}^N$ adapted to $(\F_t)_{t=1}^N$. We write $\sawu{X}_t=(X_t,\iuc{X_t})$. If the filtered process defined by
$$
\fpsawu{X}:=(\Omega, \F, \P, (\F_t)_{t=1}^N, (\sawu X_t)_{t=1}^N)$$
is Markov, we call it a \emph{Markovian lift} of $\fp X$.
\end{definition}
Note that every filtered process admits a Markovian lift (see Lemma \ref{lem:ip.self.aware}). If a process $\fp{X}$ is Markovian, its filtration provides no extra information over observing its paths. Hence, all probabilistic properties of a filtered process are contained in the law of a Markov lift. More precisely, we have
\begin{theorem}
Two filtered processes $\fp X, \fp Y$ are equivalent in the sense of Hoover--Keisler if and only if they admit Markovian lifts $ \fpsawu{X}$, $ \fpsawu{Y}$ such that $\Law(\sawu X) = \Law(\sawu Y)$.
\end{theorem}

\subsection{Adapted Wasserstein distance}
Recognizing that solely analyzing the laws of stochastic processes is insufficient for deciding if they have the same probabilistic properties, it is natural that the associated weak convergence is not appropriate in many instances. For instance,  problems such as optimal stopping or pricing and hedging in mathematical finance and basic operations like the Doob decomposition are not continuous w.r.t.\ the usual weak topology / Wasserstein distance. 
These shortcomings have motivated the definition of `adapted' counterparts.  
Specifically, adapted versions of the Wasserstein distance have been independently defined in \cite{Ve70, Ve94, Ru85,Gi04, PfPi12, BiTa19,NiSu20}. Usually the definition is given using nested disintegrations or couplings that satisfy certain relations of conditional independence. 

In the present article, we give a characterization of adapted Wasserstein distance that is in the spirit of the probabilistic definition of the Wasserstein distance $\W_p$. Recall that
\begin{align}\label{eq:defW}
\W_p^p(\mu^0,\mu^1) = \inf \{\E_\P [|Y^0-Y^1|_p^p] \},
\end{align}
where $|\cdot|_p$ is the $p$-norm on $\R^d$ and the infimum is taken over all random variables $Y^0 \sim \mu^0$, $Y^1 \sim \mu^1$ that are defined on a common probability space $(\Omega,\F,\P)$. 

In order to state the respective assertion for the adapted Wasserstein distance, we still write  $|\cdot|_p$ to denote the $p$-norm on $\R^{dN}$. To avoid repeated case distinctions between weak convergence and $p$-Wasserstein convergence, we introduce the  convention $|\cdot|_0:= 
|\cdot|_1\wedge 1$.  

\begin{theorem}\label{thm:AWintro} Let $p\in \{0\}\cup [1,\infty)$.
The adapted Wasserstein distance between two filtered processes $\fp X^0, \fp X^1$ is given by\footnote{
Here and in the following, we adopt the convention that $p$-th and $1/p$-th powers are to be ignored in the case $p=0$.
}  
\begin{align}\label{eq:AW_intro}
\AW_p^p(\fp X^0, \fp X^1) =  \inf \left\{ \E_\P \left[  | Y^0 - Y^1|_p^p \right] \right\} \quad
\end{align}
where the infimum is taken over all stochastic processes $(Y^0_t)_{t=1}^N, (Y^1_t)_{t=1}^N $ that are defined on a common stochastic basis $(\Omega,\F,\P,(\F_t)_{t=1}^N)$ and satisfy $(\Omega,\F,\P,(\F_t)_{t=1}^N, Y^i) \simhk \fp X^i$ for $i \in \{0,1\}$. 

Moreover, the infimum in \eqref{eq:AW_intro} is attained and $(\Omega,\F)$ can be taken as standard Borel space.
\end{theorem}

   
   We have $\AW_p(\fp{X},\fp{Y})=0$ if and only if $\fp X$ and $\fp Y$ are equivalent in the sense of Hoover--Keisler (see Theorem~\ref{thm:equivalence}). For each $p\in \{0\}\cup [1,\infty)$, $\AW_p$ is a complete separable metric on $\FFP_p$. The topology induced by $\AW_0$ is called the  \emph{adapted weak topology}.
    As in the case of the classical Wasserstein distance, convergence w.r.t.\ $\AW_p$ is equivalent to adapted weak convergence together with the convergence of $p$-th moments, see \cite{BaBePa21}. The adapted weak topology was originally introduced by Hoover-Keisler \cite{HoKe84, Ho91} as the topology induced by weak convergence of iterated prediction processes.



Theorem \ref{thm:AWintro} asserts that 
for the appropriate choice of the stochastic base $(\Omega, \F, \P, (\F_t)_{t=1}^N)$ and random variables $Y^0, Y^1$ on it,  we have 
\begin{align}\label{eq:AWasLP}
\AW_p(\fp X^0, \fp X^1)=\|Y^0-Y^1\|_{L^p ( \P)}.
\end{align}
Importantly, this representation of the adapted Wasserstein distance yields that $(\FFP_p,\AW_p)$ is a geodesic space for all $p \in [1,\infty)$ and provides the following representation of geodesics: A geodesic between the filtered processes $\fp X^0$ and $\fp X^1$ is given by
$$
\fp Y^\lambda :=(\Omega,\F,\P,(\F_t)_{t=1}^N, (1-\lambda) Y^0 + \lambda Y^1), \quad \lambda \in [0,1].
$$


To provide additional context we recall some further properties of $\AW_p$:
\begin{itemize}
\item the set of martingales is closed and geodesically convex for $p \ge 1$ (\cite[Theorem~1.4]{BaBePa21})
\item there is a Prohorov-type compactness criterion (\cite[Theorem~1.7]{BaBePa21}) 
\item finite state Markov chains are dense (\cite[Theorem~5.4]{BaBePa21})
\item The following are (Lipschitz-) continuous w.r.t.\ to $\AW_p$: Doob-decomposition, optimal stopping, snell-envelope, pricing and hedging, utility maximization, see \cite{BaBaBeEd19a, BaBePa21}.
\end{itemize}


In general, it is not possible to represent more than two processes on the same stochastic basis in such a way that \eqref{eq:AWasLP} is satisfied for all pairs of processes -- this is already impossible in the special case $N=1$ where $\AW$ is the classical Wasserstein distance, see \cite[Example~7.3.3]{AmGiSa08}.

However, the proof of Theorem~\ref{thm:AWintro} yields the following:  Given $\fp X, \fp X_1, \fp X_2,\ldots  \in \FP_p$, $p\in \{0\}\cup [1,\infty)$  there are $Y, Y_1, Y_2,\ldots $ on a common stochastic basis $ (\Omega,\F,\P,(\F_t)_{t=1}^N)$
 such that $\fp X \simhk \fp Y,$ $\fp X_n \simhk \fp Y_n, n\geq 1$ and
 \begin{align}\label{eq:Skorohod_Isometry} 
     \AW_p( {\fp X}, {\fp  X}^n) 
     = \|Y-Y^n\|_{L^p(\lambda^N)}, \quad n\geq 1.
 \end{align}
As a consequence we have the following Skorohod-type characterization of $\AW_p$-convergence:
\begin{corollary}
    For $\fp X, \fp X^1, \fp X^2, \ldots \in \FP_p$ the following are equivalent:
	\begin{enumerate}[label=\emph{(\alph*)}]
		\item $\AW_p(\fp X, \fp X^n) \to 0$.
		\item There exist filtered processes $ {\fp Y}\simhk\fp X$, $ {\fp  Y}^n\simhk {\fp  X}^n, n\geq 1$ on a common stochastic basis such that $\| Y-Y^n\|_{L^p(\lambda^N)}\to 0$. 
	\end{enumerate}
\end{corollary}

Classically, the Skorokhod representation theorem is formulated for almost sure convergence and this also holds in our context. Indeed we construct for each $\fp X\in \FP$ a $\simhk$-representative 
$$
\underline {\fp X} = ([0,1]^N, \B_N,  \lambda^N,  (\B_t)_{t=1}^N, (\underline X_t)_{t=1}^N), \quad \text{where } \B_t= \sigma ((x_1, \ldots, x_N) \mapsto (x_1, \ldots, x_t)) 
$$ 
such that following are equivalent for $\fp X, \fp X_1, \fp X_2, \ldots \in \FFP$: 
\begin{enumerate}[label=\emph{(\alph*)}]
    \item $\fp X^n\to \fp X$ in the weak adapted topology. 
    \item $\underline X^n \to \underline X$ in $L^0(\lambda^N)$.
        \item $\underline X^n \to \underline X$,  $\lambda^N$-a.s.
\end{enumerate}
That is, via the mapping $\fp X \mapsto \underline X$, $\FFP$ corresponds to a closed subset of $L_0([0,1]^N)$. For $p\geq 1$, $\FFP_p$ corresponds then to the subset of those $\underline X$ which are also $p$-integrable.

\subsection{Organization of the paper}
\underline{Section 2} reviews some related literature,  provides the necessary background on adapted Wasserstein distances and recalls known results on the space of filtered processes which we will require subsequently. 
In \underline{Section 3} we discuss Markovian lifts and the connection to the equivalence of filtered processes as well as to adapted weak convergence. In probability one often extends the base space to include further needed objects such as independent uniform random variables. 
In \underline{Section 4} we provide the appropriate notion of extension for filtered probability spaces. 
This is used in \underline{Section 5} to provide a \emph{transfer principle} for filtered processes.
The classical transfer principle  (see e.g.\ \cite[Theorem 5.10]{Ka97}) asserts that for  random variables $X,Y$ on a probability space $(\Omega,\F,\P)$ and a random variable $\widetilde{X} \sim X$ on a further (sufficiently large) space $(\widetilde{\Omega},\widetilde{\F},\widetilde{\P})$, there is a random variable $\widetilde{Y}$ on this space such that $(X,Y) \sim (\widetilde{X},\widetilde{Y})$. The analogous statement holds for filtered processes, see Theorem \ref{thm:transfer}. An important consequence is that optimization problems are equivalent for processes which are $\simhk$-equivalent. 
The transfer principle is a main ingredient in \underline{Section 6} where we establish the characterization of adapted Wasserstein distance mentioned in the introduction and provide a representation of couplings in terms of (bi-) adapted processes. We briefly explain how this representation  leads to simplified proofs of stability results w.r.t.\ adapted Wasserstein distance. 
In \underline{Section 7} we establish $L^p$- and a.s.-versions of the Skorokhod representation theorem based the transfer principle and results from \cite{Bo07}.
In the \underline{\smash{appendix}} we consider an amalgamation result of Hoover \cite{Ho92} which would imply our transfer principle. However this amalgamation theorem is flawed as witnessed by Example \ref{ex:amalgamation}.

\section{Background}
We start with a review of the related literature in Subsection~\ref{sec:literature}. Afterwards, we introduce notation in Subsection~\ref{sec:notation} and summarize necessary background from \cite{BaBaBeEd19b} and \cite{BaBePa21} in the Subsections~\ref{sec:recap.canon} and \ref{sec:recap.FP}.

\subsection{Remarks on related literature.} 
\label{sec:literature}
As noted above, topologies (or distances) that account for the flow of information were introduced by different groups of authors. This includes Aldous in stochastic analysis \cite{Al81}, 
Hoover--Keisler in mathematical logic \cite{HoKe84, Ho91}, Hellwig in economics \cite{He96}, Pflug--Pichler \cite{PfPi12, PfPi14} in optimization, Bion-Nadal--Talay and Lasalle \cite{BiTa19, La18} in stochastic control / stochastic analysis, Nielsen--Sun \cite{NiSu20} in machine learning and, using higher rank signatures, Bonner, Liu, and Oberhauser \cite{BoLiOb23}. As processes $(X_t)_{t=1}^N$ with their natural filtration $\F_t= \sigma( (X_k)_{k=1}^t), t\leq N$ correspond to laws in $\mathcal P((\R^d)^N)$ all of these approaches provide topologies on $\mathcal P((\R^d)^N))$. It is easy to see that these are strict refinements of the usual weak topology. More surprisingly, as established in \cite{BaBaBeEd19b, Pa22}, all of them provide \emph{the same} topology on $\mathcal P((\R^d)^N)$ which motivates to refer to it as \emph{the adapted weak topology}.
In contrast, 
the choice of an adapted topology / distance in continuous time setups appears to be more specific to the particular problem at hand, see e.g.\ \cite{BaBaBeEd19a, AcBaZa20, BiTa19, Fo22a, Fo22b} as well as the discussion in \cite[Appendix C]{BaBePa19}.

\medskip

Adapted topologies and adapted transport theory have been applied to geometric inequalities \cite{BoKoMe05, BaBeLiZa17, La18, Fo22a}, stochastic optimization and multistage programming \cite{Pf09, PfPi12, KiPfPi20, BaWi23}, they are a useful tool for different problems in mathematical finance, see e.g.\ \cite{Do14} (pricing of game options), \cite{AcBaZa20} (questions of insider trading and enlargement of filtrations), \cite{GlPfPi17, BaDoDo20, BaBaBeEd19a, BeJoMaPa21a} (stability of pricing / hedging and utility maximization) and \cite{AcBePa20} (interest rate uncertainty) and appear in machine learning, see e.g.\ \cite{AkGaKi23} (causal graph learning), \cite{XuLiMuAc20, AcKrPa23} (video prediction and generation), and \cite{XuAc21} (universal approximation). We refer to \cite{PfPi16, BaBaBeWi20, AcHo22} for results concerning the estimation of $\mathcal{AW}_p$ from statistical data and to \cite{PiWe21, EcPa22, BaHa23} for efficient numerical methods for adapted transport problems and adapted Wasserstein distances.

\medskip

Most closely related to the present article, Hoover \cite{Ho91} provides a Skorokhod-representation theorem for the weak adapted topology. Hoover works in a somewhat different setup, specifically he considers stochastic bases equipped with a  continuous time filtration which carries a single random variable rather than an adapted processes. To accustom for the  continuous time framework, it is necessary to allow the representing random variables in the Skorokhod result to have different filtrations.  The representation theorem in continuous time relies on the flawed amalgamation theorem (cf.\ Example \ref{ex:amalgamation}). Importantly, for discrete time filtrations, Hoover requires neither the consideration of different filtrations nor the use of the amalgamation  theorem and establishes \cite[Theorem~10.1]{Ho91} which is closely related to our pointwise Skorokhod representation result in Theorem~\ref{thm:asskorokhodmap}. Specifically, Hoover is concerned with filtered random variables rather than processes and his representation is specific to sequences. 



Also related to the present work is \cite{BePaPo21} where it is shown that on $\mathcal{P}((\R^d)^N)$, the weak adapted topology is metrized by the Knothe--Rosenblatt distance induced by multivariate extensions of quantile functions, so called quantile processes. As a consequence, these quantile processes provide an explicit Skorokhod representation theorem on $\mathcal{P}((\R^d)^N)$.

\subsection{Notation}
\label{sec:notation}
In all of what is to come, we fix the number of time steps $N\in\N$ and $p \in [1,\infty)$. We will omit the convention for $p=0$ further on, and work with a bounded/truncated metric instead. We write $\mathcal{P}(\X)$ for the Borel probability measures on a Polish space $\X$ and given a compatible metric $d$ on $\X$ we write $\mathcal{P}_p(\X)$ for the subset of probabilities with finite $p$-th moment. 
The Lebesgue measure on $[0,1]^N$ is denoted by $\lambda^N$.


In contrast to the introduction we do not restrict to processes filtered process with values in $\R^d$, but we allow $X_t$ to take values in the general Polish space $\X_t$. In this case we write  $\mathcal X:=\prod_{t=1}^N\X_t$ and if $d_t$ are metrics for $\X_t$, we equipp $\X$ with the metric $ d(x,y)^p=\sum_{t=1}^Nd_t(x_t,y_t)^p$.

If $\fp X$ is a filtered process, we generally refer to the elements of the 5-tuple as 
\begin{align*}
\fp{X}= (\Omega^\fp{X}, \F^\fp{X}, \P^\fp{X}, (\F^\fp{X}_t)_{t=1}^N, (X_t)_{t=1}^N).
\end{align*}
We use the convention $\F^{\fp X}_0:=\{\emptyset,\Omega^{\fp X} \}$. Given two filtered processes $\fp X, \fp Y$, we  denote the product $\sigma$-algebra $\F_s^\fp{X}\otimes\F_t^\fp{Y}$ by $\F_{s,t}^{\fp{X},\fp{Y}}$for $s,t \in \{1,\dots, N\}$ and identify $\mathcal{F}^{\fp X,\fp Y}_{0,t}$ with $\F^{\fp Y}_t$. 

If $X$ is a random variable, we write $\sigma(X)$ for the $\sigma$-algebra generated by $X$. By slight abuse of notation, we also write $\sigma(X)$ for the filtration generated by a stochastic process $X=(X_t)_{t=1}^N$, i.e.\ $\sigma(X)_t = \sigma(X_t)$. 

For $1\le r\le s \le t \le N$,  we introduce the following short-hand notations: Given $x_s \in \X_s$, we write $x_{r:t} = (x_r,\dots, x_t)$ and given $M_s \subseteq \X_s$, we write $M_{r:t} = M_r \times \dots \times M_t$. Moreover, $\pj_s : \X_{r:t} \to \X_s$ denotes the projection.

\subsection{The adapted Wasserstein distance on $\mathcal{P}_p(\X)$}\label{sec:recap.canon}

 Given a compatible complete metric $d$ on the Polish space $\X$, the Wasserstein distance between $\mu,\nu\in\Pc_p(\mathcal{\X})$ is given by
\[	
\W_p(\mu,\nu):= \inf_{\pi \in \scpl(\mu,\nu)} \left\{\int d(x,y)^p \, \pi(dx,dy)\right\}^\frac{1}{p},
\]
where $\scpl(\mu,\nu)$ is the set of couplings between $\mu$ and $\nu$, i.e.\ the set of all $\pi \in \mathcal{P}(\X \times \X)$ with first marginal $\mu$ and second marginal $\nu$. The Wasserstein distance $\W_p$ is a complete metric on probability measures that turns them into a Polish space. Convergence in $\W_p$ is equivalent to weak convergence plus convergence of the $p$-th moments, see e.g.\ \cite[Chapter 5]{Vi09}. 

As already outlined in the introduction, many relevant properties of stochastic processes are not stable w.r.t.\ the usual Wasserstein distance. This motivates to consider a stronger distance on $\mathcal{P}(\X)$, namely the adapted Wasserstein distance. This distance is defined by restricting the admissible couplings to those which respect the flow of information. Formally, this leads to the concept of bicausal couplings. 
\begin{definition}[Bicausal couplings of laws on the path-space]\label{def:bc_canonical}
	Let $\mu\in\mathcal{P}(\X)$ and $\nu\in\mathcal{P}(\Y)$. A coupling
	$\pi\in\scpl(\mu,\nu)$ is bicausal if under $\pi$ for every $t \le N$
\[
Y_{1:t} \indep_{X_{1:t}} X  \qquad \text{and} \qquad X_{1:t} \indep_{Y_{1:t}} Y,
\]
where $X_t : \X \to \X_t$ and $Y_t : \Y \to \Y_t$ are the coordinate processes. 	
\end{definition}

\begin{definition}\label{def:AW_canonical}
The adapted Wasserstein distance between $\mu, \nu \in \mathcal{P}(\X)$ is defined as
\[
\AW_p(\mu,\nu) = \inf_{ \pi \in \scplbc(\mu,\nu) } \left\{ \int d(x,y)^p \pi(dx,dx) \right\}^{\frac{1}{p}}.
\]   
\end{definition}

While the adapted Wasserstein distance has many desirable properties (e.g.\ stability of optimal control problems), it is not a complete distance on $\mathcal{P}(\mathcal{X})$. 

\subsection{The space of filtered processes}\label{sec:recap.FP}
In this section we formally introduce the space of filtered process $\FFP$, which turns out to be the completion of $\mathcal{P}(\X)$ w.r.t.\ the adapted Wasserstein distance. Recall that a filtered process is a five tuple
\[ 
\fp{X}= (\Omega, \F, \P, (\F_t)_{t=1}^N, (X_t)_{t=1}^N),
\]
where $X =(X_t)_{t=1}^N$ is a stochastic process that is adapted to $(\F_t)_{t=1}^N$. In contrast to Definition~\ref{FPDef} in the introduction we do not restrict to processes with values in $\R^d$, but we allow $X_t$ to take values in the general Polish space $\X_t$, as this turns out to be technically more convenient for our purposes. 
 
In order to define the adapted Wasserstein distance between filtered processes, we need to introduce the notion of (bi)causal couplings between filtered processes. 

\begin{definition}[Causal couplings, \protect{\cite[Definition 2.1]{BaBePa21}}]Let $\fp{X}$ and $\fp{Y}$ be filtered processes. 
A coupling $\pi$ between $\fp{X}$ and $\fp{Y}$ is a probability measure on $\Omega^{\fp X} \times \Omega^{\fp Y}$ with marginals $\P^{\fp X}$ and $\P^{\fp Y}$. It is called
\label{def:causalcoup} 
\begin{enumerate}[label=(\roman*)]
\item  \emph{causal} from  $\fp{X}$ to $\fp{Y}$ if $\F^
{\fp{X},\fp{Y}}_{N,0}$ and $\F^
{\fp{X},\fp{Y}}_{0,t}$ are conditionally independent given $\F^
{\fp{X},\fp{Y}}_{t,0}$  under $\pi$ for every $t \in \{1, \dots, N \}$.
 \item  \emph{bicausal} if it is causal from  $\fp{X}$ to $\fp{Y}$ and causal from  $\fp{Y}$ to $\fp{X}$.
\end{enumerate}

We use the notation $\cpl(\fp{X},\fp{Y})$ for the set of couplings, $\cplc(\fp{X},\fp{Y})$ and $\cplbc(\fp{X},\fp{Y})$ for the set of causal and bicausal couplings, respectively.
\end{definition}
Note that the notion of (bi-)causal couplings can also be used for stochastic bases. The following criterion for causality will be useful throughout this article. 

\begin{lemma}\label{lem:causal_eqiv}
For $\pi \in \cpl(\fp X, \fp Y)$ the following are equivalent:
 	\begin{enumerate}[label=(\roman*)]
 		\item \label{it:causal CI}$\pi$ is causal from $\fp X$ to $\fp Y$.
 		\item \label{it:causal CI1}
 		$\E_\pi[ U | \F^{\fp X,\fp Y}_{t,t} ]
 		= \E_\pi [ U | \F^{\fp X,\fp Y}_{t,0} ]$ for all $1\leq t \leq N$ and bounded $\F_N^\fp X$-measurable $U$.
 		\item \label{it:causal CI2}
 		$\E_\pi[ V | \F^{\fp X,\fp Y}_{N,0} ]
 		= \E_\pi[ V | \F^{\fp X,\fp Y}_{t,0} ]$
 		for all $1\leq t\leq N$ and bounded $\F_t^\fp Y$-measurable $V$.
 	\end{enumerate}
\end{lemma}

We are now ready to define the adapted Wasserstein distance on $\FP_p$. 
\begin{definition}[Adapted Wasserstein distance, \protect{\cite[p8]{BaBePa21}}] \label{def:AW}
The \emph{adapted Wasserstein distance} between $\fp X, \fp Y \in \FP_p$ is defined by
\begin{align*}
\label{eq:AWdef}
\AW_p^p(\fp{X},\fp{Y}):=\inf_{\pi\in\cplbc(\fp X, \fp Y) } \E_{\pi}[d(X,Y)^p].
\end{align*}
\end{definition}
It is shown in \cite[Theorem 4.11]{BaBePa21} that the equivalence relation induced by identifying processes $\fp X, \fp Y$ if $\AW_p(\fp X, \fp Y)=0$ coincides with the equivalence relation of `having the same probabilistic properties' defined by Hoover--Keisler. Therefore, the space of filtered processes $\FFP_p$ is the factor space
\[
\FFP_p = \FP_p /_{\AW_p}.
\]

The information process  of a filtered process  provides a canonical way to capture the information in filtrations that is relevant for the process at each time step. It takes values in the canonical space, which is defined as follows.
 \begin{definition}[Canonical space, \protect{\cite[Definition 3.1]{BaBePa21}}]
	\label{def:canonical.space.Z}
 Starting with $(\Z_N,d_{\Z_N}):=(\X_N,d)$ we define recursively
 \begin{align*} 
 	\Z_t := \Z_t^-\times\Z_t^+ :=\X_t\times\Pc_p(\Z_{t+1})
 \end{align*}
 for $t=N-1,\dots,1$ and equip this space with the metric $d_{\Z_t}^p:= d_t^p + \W_{p}^p$.
The \emph{canonical filtered space} is given by the triplet
	\begin{equation} \label{eq:CFS}
		\left( \Z, \F^\Z, (\F^\Z_t)_{t=1}^N \right),
	\end{equation}
	 where $\Z:=\Z_{1:N}$, $\F^\Z$ is the Borel-$\sigma$-field on $\Z$ and $\F_t^\Z$ is the $\sigma$-algebra generated by the projection $\pj_{1:t}:\Z\rightarrow \Z_{1:t}$.
We define the following evaluation maps on canonical filtered spaces
\begin{align}\label{def:evaluation map}
Z^- \colon \Z \to \X_{1:N}  \text{  with  } Z^-(z) &:= (Z^-_t(z))_{t = 1}^N := (z_t^-)_{t = 1}^N\\
Z^+ \colon \Z \to \Z^+_{1:N}  \text{  with  } Z^+(z) &:= (Z^+_t(z))_{t = 1}^N := (z_t^+)_{t = 1}^N.
\end{align}  
\end{definition}

\begin{definition}[The information process, \protect{\cite[Definition 3.2]{BaBePa21}}]
	\label{def:ip}
	Let $\fp X$ be a filtered process. Its \emph{information process} $\ip(\fp X)=(\ip_t(\fp X))_{t=1}^N$ is defined recursively by	$\ip_N(\fp X)=X_N$ and
		\begin{align*}
		\ip_t(\fp X) =(\ip_t^-(\fp{X}), \ip_t^+(\fp{X})) \colon\Omega^\fp X \to \mathcal{Z}_t, \\
		\omega \mapsto \left( X_t(\omega), \Law(\ip_{t+1}(\fp{X}) | \F_t^\fp{X})(\omega) \right)
		\end{align*}
  for $t<N$.
\end{definition}
Noting that the spaces introduced in Definition \ref{def:canonical.space.Z} are Polish, the conditional probabilities in the definition above exist. 

These concepts yield a canonical way to represent $\fp X \in \FP_p$ by equipping the canonical filtered space from \eqref{eq:CFS} with the law of $\ip(\fp X)$ and considering the evaluation map $Z^-$  introduced in \eqref{def:evaluation map}.
\begin{definition}[Associated canonical filtered processes, \protect{\cite[Definition 3.8]{BaBePa21}}]
		\label{def:CFP.associated.to.FP}
		The \emph{canonical filtered process  associated to} a filtered process  $\fp X$ is given by
		\begin{align*}
			\widehat{\fp X} = \left( \Z, \F^\Z, \Law(\ip(\fp X)),(\F^\Z_t)_{t = 1}^N,  Z^- \right).
		\end{align*}
		\end{definition}
Note that we have $\AW(\fp X, \widehat{\fp X})=0$, i.e.\ $\widehat{\fp X}$ is indeed a representative of the equivalence class of $\fp X$. Moreover, the canonical filtered process is based on a standard Borel probability space.

\begin{theorem}[Isometry, \protect{\cite[Theorem 3.10]{BaBePa21}}] \label{thm:isometry}
 For $\fp X$ and $\fp Y \in \FP_p$ and their associated canonical processes $\widehat{\fp X}$ and $\widehat{\fp Y}$ it holds that
	\begin{align} \label{eq:isometry}
		\AW_p\left( \fp X, \fp Y \right) = \AW_p( \widehat{\fp X}, \widehat{\fp Y} ) = \W_p\left( \Law(\ip_1(\fp X)), \Law(\ip_1(\fp Y)) \right).
	\end{align}
 Moreover, $(\FFP_p,\AW_p)$ is a complete metric space that is isometrically isomorphic to $(\Z_1,\W_p)$.
\end{theorem}

\begin{remark}
Let $\fp X, \fp Y \in \FP_p$. It is well known from standard optimal transport theory (see e.g.\ \cite[Theorem 1.3]{Vi03}) that the infimum in $\W_p\left( \Law(\ip_1(\fp X)), \Law(\ip_1(\fp Y)) \right)$ is attained, i.e.\ that there is an optimal coupling $\widehat{\pi} \in \scpl(\Law(\ip_1(\fp X)),  \Law(\ip_1(\fp Y)))$. 

Provided that the stochastic bases of $\fp X, \fp Y$ are standard Borel, the arguments given in the proof of \cite[Theorem 3.10]{BaBePa21} show that this coupling  $\widehat{\pi}$ induces an optimal coupling $\pi \in \cplbc(\fp X, \fp Y)$. Hence, if the stochastic bases of $\fp X, \fp Y$ are standard Borel, the infimum in $\AW_p(\fp X, \fp Y)$ is attained. 
\end{remark}

We conclude with a short explanation how the definitions of adapted Wasserstein distance and bicausal coupling given in this subsection are compatible with those given in Subsection~\ref{sec:recap.canon}. 

\begin{remark}\label{rem:embed}
A distribution of the path space $\mu \in \PP(\X)$ naturally induces a filtered process, namely the coordinate process with the canonical filtration
\[
\mathbb{S}^\mu := ( \X, \B_\X, \mu, \sigma(X_{1:t})_{t=1}^N, X),
\]
where $X_t : \X \to \X_t$ is the coordinate processes. By comparing Definition~\ref{def:bc_canonical} and Definition~\ref{def:causalcoup}, it is easy to see that 
\[
\cplbc(\mathbb{S}^\mu, \mathbb{S}^\nu) = \scplbc(\mu,\nu).
\]
Therefore, the map 
\[
\PP(\X)  \to \FFP_p : \mu \mapsto \mathbb{S}^\mu
\]
is an isometric embedding of $\PP(\X)$ equipped with the distance  $\AW_p$ (cf.\ Definition~\ref{def:AW_canonical})  into  the factor space $\FFP_p = \FP_p/_{\AW_p}$ equipped with the distance $\AW_p$ (cf.\ Definition~\ref{def:AW}). It is shown in \cite[Theorem~1.2]{BaBePa21} that the range of this embedding is dense, i.e.\ that $(\FFP_p,\AW_p)$ is the metric completion of $(\PP(\X),\AW_p)$.     
\end{remark}


\section{Self-aware Lifts and Equivalence of Stochastic Processes}\label{sec:salift}
The aim of this section is to prove the characterization of the Hoover--Keisler equivalence in terms of lifts. As it appears to be natural in this context, we lift the filtered processes to so-called self-aware processes. At the end of this section, we briefly discuss how to describe convergence in adapted distribution using lifts.  



\begin{definition}[Self-aware process]\label{def:sa}
A filtered process $(\Omega,\F,\P,(\F_t)_{t=1}^N,({X}_t)_{t=1}^N)$ is called \emph{self-aware} if for every $t\leq N$
\begin{align}\label{eq:sa}
	\Law(X | \F_t)=\Law(X | X_{1:t}).
\end{align}
\end{definition}
Note that \eqref{eq:sa} can be interpreted as the following conditional independence: Given $X_{1:t}$, the process $X$ is independent of $\F_t$. Hence, self-aware processes are precisely those processes, for which the filtration carries no extra information about the paths $X$. This is fact is made precise in Corollary~\ref{cor:sa} at the end of this section, which states the self-aware processes are precisely those processes that are $\simhk$-equivalent to a process with canonical filtration. 

\begin{definition}[Self-aware lift]\label{def:sa_lift}
Let $\fp{X}=(\Omega, \F, \P, (\F_t)_{t=1}^N, (X_t)_{t=1}^N)$ be a filtered process and $\iuc{X}= (\iuc{X_t})_{t=1}^N$ be adapted to $(\F_t)_{t=1}^N$. We write $\sawu{X}_t=(X_t,\iuc{X_t})$. If the filtered processes defined by
$$
\fpsawu{X}:=(\Omega, \F, \P, (\F_t)_{t=1}^N, (\sawu X_t)_{t=1}^N)$$
is self-aware, we call $\sawu{X}$ a \emph{self-aware lift} of $\fp X$. 
\end{definition}
\begin{remark}\label{rem:sa_vs_Markov}
It is easy to see that every Markov process is self-aware. Conversely, if $(X_t)_{t=1}^N$ is self-aware, then the process $(X_{1:t})_{t=1}^N$ is Markov. In light of this observation, self-aware lifts are only slight generalization of Markov lifts. If $\sawu{X}=(X, \iuc{X})$ is a self-aware lift of $\fp X$, then process defined by $\sawu{X}'_t := (X_t, (X_{1:t}, \iuc{X_{1:t}})) $ is a Markov lift of $\fp X$. 
\end{remark}


\begin{lemma}[\protect{\cite[Lemma 3.3]{BaBePa21}}]
\label{lem:ip.self.aware}
For every $\fp{X} \in \FP$, the information process $\ip(\fp{X})$ is a self-aware lift of $\fp{X}$.
\end{lemma}

It turns out that the information process is not an arbitrary self-aware lift, but it has a certain minimality property. To make this precise, we need the following

\begin{definition}[Minimal self-aware lift]
\label{def:min_sa_lift}
 Let $\wcheck{X}=(X,\iuc{X})$ be a self-aware lift of a filtered process $\fp{X}$. Then $\wcheck{X}$ is  called \emph{minimal self-aware lift} if for every further self-aware lift $\wcheck{X}'=(X,\iuc{X'})$ of $\fp{X}$ there exists an adapted\footnote{In the present setting, this means that $T$ is of the form  $T(x,\iuc{x}') = (T_1(x_1, \iuc{x_1}'), T_2(x_{1:2}, \iuc{x_{1:2}}'), \dots T_n(x,\iuc{x}'))$.  } map $T:\sawu{\X}' \to \sawu{\X}$ such that $T(\wcheck{X}')=\wcheck{X}$.
\end{definition}



\begin{proposition}
\label{prop:sa}
For every $\fp X \in \FP$, the information process $\ip(\fp X)$ is a minimal self-aware lift.
\end{proposition}
\begin{proof}
As $\ip(\fp X)$ is a self-aware lift of $\fp X$ by Lemma \ref{lem:ip.self.aware}, it suffices to show that for any further self-aware lift $\sawu{X}$, there exists an adapted map $T = (T_t)_{t=1}^N$ such that  $\ip_t(\fp{X})=T_t(\sawu{X}_ {1:t})$ for every $t \le N$. We aim to construct $T$ by backward induction. 

For $t=N$, we set $T_N := \pj_{\X_N}$. Indeed, we have  $\ip_N(\fp{X})=X_N=\pj_{\X_N}(\sawu{X})$.

Let $t <N$. Assume that there is a measurable  map $T_t:\sawu\X\to\Z_t$ satisfying $\ip_t(\fp{X})=T_t(\sawu{X}_ {1:t})$. Using the recursive definition of the information process in the first equality and that $\sawu{X}$ is self-aware, we obtain
  \begin{align*}
  \ip_{t-1}(\fp{X})&=(X_{t-1},\Law(\ip_{t}(\fp{X})|\F_{t-1}))\nonumber\\&=(X_{t-1},\Law(T_t(\sawu{X}_{1:t})|\F_{t-1}))\nonumber\\&=(X_{t-1},\Law(T_t(\sawu{X}_{1:t})|\sawu{X}_{1:t-1})).
  \end{align*}
This shows that $\ip_{t-1}(\fp{X})$ is $\sigma(\sawu{X}_{1:t-1})$-measurable.
Therefore, there exists a Borel map $T_{t-1} : \sawu{\X} \to \Z_{t-1}$ satisfying $\ip_{t-1}(\fp{X})=T_{t-1}(\sawu{X}_{1:t-1})$, see e.g.\  \cite[Lemma 1.13]{Ka97}.
\end{proof} 
There is a close connection between self-aware lifts and so-called self-contained filtrations. To make this precise, we first give a definition of self-contained filtrations (which is a translation of \cite[Section~1]{Ho92} to the setting of discrete time processes).
 
\begin{definition}
Let $\fp X = (\Omega,\F,\P, (\F_t)_{t=1}^N, X)$ be a filtered process. 
\begin{enumerate}[label=(\roman*)]
    \item A filtration $(\G_t)_{t=1}^N$ is called \emph{subfiltration} of $(\F_t)_{t=1}^N$, if $\G_t \subset \F_t$ for every $t=1,\dots, N$.
    \item As subfiltration $(\G_t)_{t=1}^N$ is called \emph{self-contained} in $(\F_t)_{t=1}^N$, if $\F_t$ and $\G_N$ are conditionally independent given $\G_t$ for every $t=1,\dots, N$. 
    \item The \emph{intrinsic filtration} of $\fp X$, denoted by $\mathcal{I}(\fp X)$, is the smallest  self-contained subfiltration of $(\F_t)_{t=1}^N$ such that $X$ is adapted to this filtration.
\end{enumerate}
\end{definition}
Note that taking the intersection over a family of filtrations that are self-contained in $(\F_t)_{t=1}^N$ yields a filtration that is self-contained in $(\F_t)_{t=1}^N$. Hence, there exists indeed a smallest self-contained filtration such that $X$ is adapted, see \cite[Lemma~1.2]{Ho92}.

\begin{proposition}\label{prop:sa-sc}
    Let $\fp X = (\Omega,\F,\P, (\F_t)_{t=1}^N, X)$ be a filtered process. 
\begin{enumerate}[label=(\roman*)]
    \item If $Y$ is an adapted process on the stochastic base of $\fp Y$, then $\sigma(Y)$ is a subfiltration of $(\F_t)_{t=1}^N$.
    \item If $Y$ is a self-aware process on the stochastic base of $\fp Y$, then $\sigma(Y)$ is a selfcontained subfiltration of $(\F_t)_{t=1}^N$.
    \item A self-aware lift $\sawu{X}$ is  minimally self-aware if and only if $\sigma(\sawu X) = \mathcal{I}(\fp X)$.
\end{enumerate}   
\end{proposition}
\begin{proof}
The first assertion is clear.

For the second assertion, observe that the self-awareness condition \eqref{eq:sa} precisely states that $Y$ and $\F_t$ are conditionally independent given $Y_{1:t}$ for every $t$. This is  precisely the definition of $\sigma(Y)$ being self-contained in $(\F_t)_{t=1}^N$.

To prove the third claim, first note that all minimally self-aware lifts generate the same filtration. Indeed, if $\sawu{X}$ and $\sawu{X}'$ are both minimally self-aware lifts there are adapted maps $T, T'$ such that $\sawu{X} = T(\sawu{X}')$ and $\sawu{X}' = T'(\sawu{X})$.  Hence, $\sigma(\sawu{X}) = \sigma(\sawu{X}')$. 

Next, we show that the information process generates the intrinsic filtration. By the second claim, $\sigma(\ip(\fp X))$ is a self-contained subfiltration of $(\F_t)_{t=1}^N$, so $\mathcal{I}(\fp X) \subset \sigma(\ip(\fp X))$. To show the other inclusion, we need to show that $\ip(\fp X)$ is adapted to $\mathcal{I}(\fp X)$. To that end, it suffices to show the following: If a filtration $(\G_t)_{t=1}^N$ is self-contained in $(\F_t)_{t=1}^N$ and $X$ is adapted to $(\G_t)_{t=1}^N$, then also $\ip(\fp X)$ is adapted to $(\G_t)_{t=1}^N$.

We fix such a filtration $(\G_t)_{t=1}^N$ and show by backward induction on $t$ that $\ip_t(\fp X)$ is $\G_t$-measurable. This is clear for $t=N$ because $\ip_N(\fp X) = X_N$. Assume that $\ip_t(\fp X)$ is $\G_t$-measurable. As $(\G_t)_{t=1}^N$ is self-contained, $\F_{t-1}$ and $\G_t$ are conditionally independent given $\G_{t-1}$. Therefore, we find that
\[
\ip_{t-1}(\fp X)= (X_{t-1}, \Law(\ip_t(\fp X) | \F_{t-1})) = (X_{t-1},  \Law(\ip_t(\fp X) | \G_{t-1}))
\]
is  $\G_{t-1}$-measurable.
\end{proof}


We have introduced multiple notions of equivalence for stochastic processes so far. The following theorem states that they are all equivalent.
\begin{theorem}[Equivalence of filtered processes]
  \label{thm:equivalence}
  For $\fp{X}, \fp{Y} \in \FP$, the following are equivalent:
  \begin{enumerate}[label=\emph{(\alph*)}]
    \item \label{prop:equivalencelift} $\fp{X}$ and $\fp{Y}$ admit self-aware lifts which have the same law.
    \item \label{prop:equivalencemarkov} $\fp{X}$ and $\fp{Y}$ admit Markovian lifts which have the same law. 
    \item \label{prop:equivalenceHK} $\fp{X}\simhk\fp{Y}$.
    \item \label{prop:equivalenceAW} $\AW_p (\fp{X}, \fp{Y})=0$. 
  \end{enumerate}  
\end{theorem}
\begin{proof}
The equivalence between (a) and (b) is an immediate consequence of Remark~\ref{rem:sa_vs_Markov}. 

The equivalence between (c) and (d) is shown in \cite[Theorem 4.11]{BaBePa21}.

Therefore, it remains to show that (a) and (d) are equivalent. Assume that (d) is satisfied, i.e.\ $\AW(\fp X, \fp Y) = 0$. According to Lemma~\ref{lem:ip.self.aware}, $\ip(\fp X)$ and $\ip(\fp Y)$ are self-aware lifts of $\fp X$ and $\fp Y$. By Theorem~\ref{thm:isometry} we have $\W_p(\Law(\ip(\fp X)), \Law(\ip(\fp X))) = \AW(\fp X, \fp Y) =0$. Hence, (a) is satisfied as well.

Assume that (a) holds true, i.e.\ that there are  self-aware lifts $\sawu{X}, \sawu{Y}$ of $\fp X$ and $\fp Y$ respectively that satisfy $\Law(\sawu{X}) = \Law(\sawu{Y})$. As $\ip(\fp X)$ and $\ip(\fp Y)$ are minimally self-aware lifts, there exist adapted maps $f=(f_t)_{t=1}^N$ and $g=(g_t)_{t=1}^N$ such that $\ip_t(\fp X) = f_t(\sawu{X}_{1:t})$ and $\ip_t(\fp Y) = g_t(\sawu{Y}_{1:t})$ for every $t = 1,\dots, N$. 

We aim to show by backward induction on $t$ that 
\begin{align}\label{eq:prf:equiv_ind}
    \Law(\sawu{X},\ip_t(\fp{X}))=\Law(\sawu{Y},\ip_t(\fp{Y})). 
\end{align}
For $t=N$ this is clear because $ \ip_N(\fp X) =  X_N$ is a component of $\sawu{X}_N$. 

Assume that  $\Law(\sawu{X},\ip_t(\fp{X}))=\Law(\sawu{Y},\ip_t(\fp{Y}))$. As $\ip_t(\fp X) = f_t(\sawu{X}_t)$, we find using the self-awareness of $\sawu{X}$
\begin{align}
    \begin{split}\label{eq:prf:X}
     (\sawu{X}, \ip_{t-1}(\fp{X})) &= (\sawu{X},  X_{t-1},\Law( f_t(\sawu{X}_t) | \F^{\fp{X}}_{t-1})) = (\sawu{X}, X_{t-1},\Law( f_t(\sawu{X}_t) | \sawu{X}_{1:t-1})) 
     \\&=(\sawu{X}, X_{t-1},\Law(\ip_{t}(\fp{X}) | \sawu{X}_{1:t-1})).  
    \end{split}
\intertext{By the very same argument, we also have }\label{eq:prf:Y}
     (\sawu{Y},\ip_{t-1}(\fp{Y})) &= (\sawu{Y}, Y_{t-1},\Law(\ip_{t}(\fp{Y}) | \sawu{Y}_{1:t-1})).  
\end{align}
As the laws of right hand sides of \eqref{eq:prf:X} and \eqref{eq:prf:Y} only depend on $\Law(\sawu{X},\ip_t(\fp{X}))$ and $\Law(\sawu{Y},\ip_t(\fp{Y}))$, we conclude $\Law(\sawu{X},\ip_{t-1}(\fp{X}))=\Law(\sawu{Y},\ip_{t-1}(\fp{Y}))$. This proves the inductive claim.

Therefore, we conclude $\Law(\ip(\fp{X}))=\Law(\ip(\fp{Y}))$ and Theorem \ref{thm:isometry} implies $\AW_p(\fp{X},\fp{Y})=0$. 
\end{proof}

\begin{corollary}
Let $\fp X \in \FP$ and $\sawu{X}$ be a self-aware lift of $\fp X$. Then $\fp X \simhk (\Omega^{\fp X}, \F^{\fp X}, \P^{\fp X}, \sigma(\sawu{X}), X)$. 
\end{corollary}
\begin{proof}
This is an immediate consequence of Theorem~\ref{thm:equivalence} as $\sawu{X}$ is a self-aware lift of both $\fp X$ and $(\Omega^{\fp X}, \F^{\fp X}, \P^{\fp X}, \sigma(\sawu{X}), X)$.
\end{proof}

Theorem~\ref{thm:equivalence} also provides further insights about self-aware processes. In particular, we obtain that there is a simpler canonical representative in the case of self-aware processes, namely the processes of type $\mathbb{S}^\mu$, $\mu \in \PP(\X)$, which were defined in Remark~\ref{rem:embed}.  We call $\mathbb{S}^\mu $ the standard self-aware process with law $\mu$. Using this notions, we have 
\begin{corollary}\label{cor:sa}
The following assertions hold true:
\begin{enumerate}[label=(\roman*)] 
    \item Let $\fp X \in \FP$ and write $\mu =\law(X)$. Then $\fp X$ is self-aware if and only if $\fp X \simhk \mathbb{S}^\mu$. 
    \item Let $\fp X, \fp Y \in \FP$ be self-aware. Then $\fp X \simhk \fp Y$ if and only if $\law(X)=\law(Y)$. 
\end{enumerate}  
\end{corollary}
\begin{proof}
This is an easy consequence of Theorem~\ref{thm:equivalence}(a). 
\end{proof}

While it is possible to characterize equivalence of processes with self-aware / Markov lifts, this concept is not fine enough to characterize adapted weak convergence. 
To do this, we need to consider lifts which store information in a `uniform' way. 
\begin{definition}
A process $\fp X$ is called ($1$-)Lipschitz Markov, if for every $t <N$ there exists a $1$-Lipschitz map $f_t : \X_t \to \mathcal{P}(\X_{t+1})$ (where the latter space is equipped with with the $1$-Wasserstein distance w.r.t.\ $d_{t+1}$) such that $\law(X_{t+1}|\F_t)=f_t(X_t)$. 

$\sawu{\fp X}$ is a Lipschitz Markov lift of $\fp X$ if it is Lipschitz Markov and a Markov lift of $\fp X$. 
\end{definition}
From the Kantorovoich--Rubinstein theorem it is straightforward that a process is Lipschitz Markov iff for every $f: \X_{t+1} \to \R$ 1-Lipschitz, there is a $1$-Lipschitz map $g: \X_t \to \R$ such that
$$
\E[f(X_{t+1})|\F_t] = g(X_t).
$$
It is well known that while the property of being Markov is not weakly closed, being Lipschitz Markov is closed.  This is used for instance in \cite{Ke72, HiRoYo14, Lo08b} to construct Markov processes with given marginals. 

Note that the information processes is a Lipschitz Markov lift. To see this, recall from Definition~\ref{def:ip} that $\law(\ip_{t+1}(\fp X) |\F_t^{\fp X}) = \ip_t^+(\fp X)$ and note that the projection $\ip_t(\fp X) = ( \ip_t^-(\fp X),  \ip_t^+(\fp X) ) \mapsto \ip_t^+(\fp X)$ is a contraction. 

    


\begin{proposition}
For $\fp X, \fp X^1, \fp X^2, \ldots \in \FP_p$, the following are equivalent:   
\begin{enumerate}[label=(\roman*)]
    \item $\AW_p(\fp X^n, \fp X) \to 0$.
    \item There exist Lipschitz Markov lifts $\sawu{X}^n, \sawu{X}$ of $\fp X^n $ and $\fp X$ such that $\W_p(\law(\sawu{X}^n), \law(\sawu{X}) )\to 0$. 
\end{enumerate}    
\end{proposition}

\begin{proof}
First observe that implication from (i) to (ii) is an easy consequence of the isometry \eqref{eq:isometry} and the fact that the information process is 1-Lipschitz Markov.

For the reverse implication, suppose $\fp X^n, \fp X \in \FP_p$ have $L$-Lipschitz Markov lifts  $\sawu{X}^n, \sawu{X}$ such that $\W_p(\law(\sawu{X}^n), \law(\sawu{X}) )\to 0$. The crucial step is  that the collection $\{\law(\sawu{X}_n) : n \in \N \}$ is relatively compact in $(\mathcal{P}_p(\X),\AW_p)$. This follows easily from the compactness criterion for Hellwig's information topology  \cite[Theorem~1.4]{Ed19} and the fact that on the space $\mathcal{P}_p(\X)$ convergence in  $\AW_p$ is equivalent to convergence in Hellwig's information topology plus convergence of $p$-th moments, see \cite[Theorem~1.3]{BaBaBeEd19b}. 

As $\AW_p$ is stronger than $\W_p$, the sequence $(\law(\sawu{X}_n))_n$ is relatively compact in $\AW_p$ and it converges in $\W_p$ to $\law(\sawu{X})$, we also have $\AW_p(\law(\sawu{X}^n), \law(\sawu{X}) )\to 0$. As the processes $\sawu{\fp X}^n, \sawu{\fp X}$ are all self-aware, this entails $\AW_p(\sawu{\fp X}^n, \sawu{\fp X})\to 0$ (see Remark~\ref{rem:embed}). As $\AW_p(\fp X^n,\fp X) \le \AW_p(\sawu{\fp X}^n, \sawu{\fp X})$, this yields the claim.
\end{proof}

\section{Independent Randomization and Extensions}
\label{sec:ir}
Let $X,Y$ be random variables on a probability space $(\Omega,\F,\P)$ and $\widetilde{X}$ be a random variable on a further space $(\widetilde{\Omega},\widetilde{\F},\widetilde{\P})$. Provided that $(\widetilde{\Omega},\widetilde{\F},\widetilde{\P})$ is `large enough' (i.e.\ it can support a further independent uniform random variable; this is usually tacitly assumed in probability theory), there is a random variable $\widetilde{Y}$ on this space such that $(X,Y) \sim (\widetilde{X},\widetilde{Y})$. This classical fact is referred to as \emph{transfer principle}, see e.g.\ \cite[Theorem 5.10]{Ka97}.

In order to establish a suitable transfer principle for filtered processes, we need to make precise what `large enough' means in the context of stochastic processes.\footnote{A first attempt could be to consider processes which just allow for an independent uniformly distributed random variable $U_1$ that is $\F_1$-measurable. However, this approach is too weak to allow for a transfer principle, see Remark~\ref{rem:randomNecessary} below.}

\begin{definition}[Independent randomization]\label{def:ir}
A filtered process  $\fp X$ admits \emph{independent randomization} if there exists a self-aware lift $\sawu{X}$ and for all $t = 1, \dots,N$ a uniform random variable $U_t \in \F_t$ such that 
$
U_t \indep (\sawu{X},\F_{t-1}).
$
In this case we say that $U$ is an independent randomization for the self-aware lift $\sawu{X}$. 
\end{definition}

Note that if $U$ is an independent randomization for a self-aware lift $\sawu{X}$, it is also an independent randomization for every minimal self-aware lift. 

  \begin{lemma}
  \label{lem:randomization}
    Let $\fp{X}$ be a filtered process, $\sawu{X}$ be a self-aware lift of $\fp X$ and $U$ be a independent randomization for $\sawu{X}$.     
    Then  the process $(\sawu{X}_t,U_t)_{t=1}^N$ is again a self-aware lift of $\fp{X}$.
  \end{lemma}

\begin{proof}
We need to show that for every $t \le N$
\begin{align}\label{eq:prf:extendedSAlift}
\Law(\sawu{X},U | \F_t)=\Law(\sawu{X},U | \sawu{X}_{1:t},U_{1:t} ).
\end{align}
As $\sawu X$ is self-aware and $U_{t+1:N}$ is independent of $\F_t$, we have for every bounded Borel $f :\mathcal{X}_{t+1:N} \to \R$ and $g : \R^{N-t} \to \R$
\begin{align*}
\E[f(\sawu{X}_{t+1:N})&g(U_{t+1:N}) | \F_t]  = 
\E[f(\sawu{X}_{t+1:N}) | \F_t]\E[g(U_{t+1:N})] = 
\E[f(\sawu{X}_{t+1:N}) |\sawu{X}_{1:t}]\E[g(U_{t+1:N})] \\
	&=\E[f(\sawu{X}_{t+1:N})g(U_{t+1:N}) | \sawu{X}_{1:t}]=
\E[f(\sawu{X}_{t+1:N})g(U_{t+1:N}) | \sawu{X}_{1:t},U_{1:t}].
\end{align*}
For every bounded Borel $\overline{f} :\mathcal{X}_{1:t} \to \R$ and $\overline{g} : \R^{t} \to \R$, we can further calculate 
\begin{align*}
	&\E[\overline{f}(\sawu{X}_{1:t})f(\sawu{X}_{t+1:N})\overline{g}(U_{1:t})g(U_{t+1:N}) | \F_t] \\
	&=\overline{f}(\sawu{X}_{1:t})\overline{g}(U_{1:t})\E[f(\sawu{X}_{t+1:N})g(U_{t+1:N}) | \F_t] \\
	&=\overline{f}(\sawu{X}_{1:t})\overline{g}(U_{1:t})\E[f(\sawu{X}_{t+1:N})g(U_{t+1:N}) | \sawu{X}_{1:t},U_{1:t}] \\
	&=\E[\overline{f}(\sawu{X}_{1:t})f(\sawu{X}_{t+1:N})\overline{g}(U_{1:t})g(U_{t+1:N}) | \sawu{X}_{1:t},U_{1:t}],
\end{align*}
which implies \eqref{eq:prf:extendedSAlift} using a monotone class argument.
 \end{proof}
Not every filtered process $\fp X$ allows for independent randomization, yet every filtered process can be extended to a process that does. To make this precise, we first give a formal definition of extension.

\begin{definition}[Extension, \protect{\cite[Definition 0.1]{Ho92}}]
\label{def:extension}
An \emph{extension} of a stochastic basis $(\Omega, \F, \P, (\F_t)_{t=1}^N)$ is a stochastic basis $(\dot\Omega, \F\otimes \H, \Q, (\F_t\otimes \H_t)_{t=1}^N)$, where
\begin{enumerate}[label=(\roman*)]
\item $\dot\Omega=\Omega\times\Xi$ for some set $\Xi$.
\item For each $F\in\F$, $\dot F:=F\times\Xi\in\F\otimes\H$ and $\Q(\dot F)=\P(F)$.
\item $\Q$ is a causal coupling from $(\Omega,\F,\P,(\F_t)_{t=1}^N)$ to $(\Xi,\H,\Q(\Omega\times\cdot), (\H_t)_{t=1}^N)$.
\end{enumerate}
A process $X:\Omega\rightarrow\X$ can be embedded in the extension using induced process  $\dot{X}:\dot{\Omega}\rightarrow \X$, which defined as $\dot{X}(\omega,\xi)=X(\omega)$.
\end{definition}

Note that in \cite[Definition 0.1]{Ho92} the third condition is stated in terms of a conditional independence, namely for every $t\le N$, $\dot\F_N$ and $\F_t \otimes \H_t$ are conditionally independent  given $\dot\F_t$ under $\Q$. By Lemma~\ref{lem:causal_eqiv} this is equivalent to the causality condition as stated in Definition~\ref{def:extension}(3).

A crucial property of extensions is that they preserve Hoover--Keisler-equivalence.
\begin{theorem}[\protect{\cite[Theorem 2.4]{Ho92}}]\label{thm:hovext}
    Let $(\dot\Omega, \F\otimes \H, \Q, (\F_t\otimes \H_t)_{t=1}^N)$ be an extension of the stochastic basis $(\Omega, \F, \P, (\F_t)_{t=1}^N)$ of $\fp{X}$ and let $\dot{X}(\omega,\xi)=X(\omega)$ be the induced process. Then we have
    \[
    \fp{X}\simhk \dot{\fp X} :=  (\dot\Omega, \F\otimes \H, \Q, (\F_t\otimes \H_t)_{t=1}^N,\dot{X}).
    \]
\end{theorem}
\begin{proof}
	Let $\sawu{X}$ be a self-aware lift of $\fp{X}$. By Theorem~\ref{thm:equivalence}, it suffices to observe that $\dot{\protect{\sawu{X}}}$ is a self-aware lift of $\dot{\fp X}$ that has the same law as ${\wcheck{X}}$. Indeed, using the self-awareness of $\sawu{X}$ and property (3) of extensions, we find
 \[
 \Law_{\Q}(\dot{\widecheck{X}}\vert\sawu{X}_{1:t})=\Law_{\P}(\sawu{X}\vert\sawu{X}_{1:t})=\Law_{\P}(\sawu{X}\vert\F_{t})=\Law_{\Q}(\dot{\sawu{X}}\vert\F_{t}\otimes\H_t). \qedhere
 \]
\end{proof}

We close this section with an explicit construction that shows that every $\simhk$-class contains a representative that allows for independent randomization. 
 
 \begin{proposition} \label{prop:eqivir}
 For every filtered process  there is an equivalent process that is based on a standard Borel probability space and allows for independent randomization. 
 \end{proposition}
 \begin{proof}
Let $\fp X \in \FP$. By passing to associated canonical filtered process, cf. Definition~\ref{def:CFP.associated.to.FP}, we can assume that $\fp X$ is  based on a standard Borel probability space. 
 
We set $\Xi=[0,1]^N$, $\dot \Omega = \Omega \times \Xi$, $\Hs=\B([0,1]^N)$, and $\Hs_t=\sigma(\pj_{1:t})$, where $\pj_{1:t}:[0,1]^N\rightarrow[0,1]^t$. Further, let $\Q=\P\otimes\lambda^N$. It is easy to check that $(\dot\Omega, \F\otimes \H, \Q, (\F_t\otimes \H_t)_{t=1}^N)$ is an extension according to Definition~\ref{def:extension}.  Theorem~\ref{thm:hovext} yields $\dot{\fp X} :=  (\dot\Omega, \F\otimes \H, \Q, (\F_t\otimes \H_t)_{t=1}^N,\dot{X}) \simhk \fp X$. 

Next, we observe that $\dot{\fp X} :=  (\dot\Omega, \F\otimes \H, \Q, (\F_t\otimes \H_t)_{t=1}^N,\dot{X})$  allows for independent randomization. Indeed,  $U_t:\Omega\times[0,1]^N\rightarrow[0,1] : (\omega, \xi_1, \dots ,\xi_N)\mapsto \xi_t$, is uniformly distributed and $\F_t\otimes\Hs_t$-measurable. It is easy to check that $U_t$ is independent of $(\F_{t-1}\otimes \H_{t-1}, \ip(\fp{X}))$ for every $t$.  
\end{proof}


\section{Process couplings and the transfer principle}
The intuition behind bicausal couplings is that they allow to consider two filtered processes $\fp X^1, \fp X^2$ as processes on a common stochastic basis. In the first part of this section, we will make this intuition precise. Afterwards, we derive a so-called transfer principle for filtered processes.

\begin{definition}
For $\fp{X^1}, \fp{X^2} \in \FP$ we define the set of respective \emph{process couplings} as those vector-valued process, whose component processes are $\simhk$-equivalent to $\fp{X^1}, \fp{X^2}$, i.e.\ 
\begin{align*}
	\fpcpl(\fp X^1, \fp X^2) := \{(\overline{\Omega}, \overline{\F}, \overline{\P}, (\overline{\F}_t)_{t=1}^N, (\overline{X}{}_t^1, \overline{X}{}^2_t)_{t=1}^N): (\overline{\Omega}, \overline{\F}, \overline{\P}, (\overline{\F}_t)_{t=1}^N, (\overline{X}{}^i_t)_{t=1} ^N)\simhk \fp {X^i} \text{ , } i\in \{1,2\}\}.
		\end{align*}
We write $\fpcpl(\fp X^1, \fp X^2)$ for the collection of all processes couplings. 
\end{definition}
A convenient way to construct process couplings is the following: 
we can consider  two processes $\overline{\fp{X}}{}^1,  \overline{\fp{X}}{}^2$ on the same stochastic basis  as one vector valued process
\[
\overline{\fp{X}}{}^1\marriage \overline{\fp{X}}{}^2 := (\overline{\Omega}, \overline{\F}, \overline{\P}, (\overline{\F}_t)_{t=1}^N, (\overline{X}{}_t^1, \overline{X}{}^2_t)_{t=1}^N).
\]
If $\overline{\fp{X}}{}^1 \simhk \fp X^1$ and $  \overline{\fp{X}}{}^2 \simhk \fp X^2$, we have $\overline{\fp{X}}{}^1\marriage \overline{\fp{X}}{}^2 \in \fpcpl(\fp X^1, \fp X^2)$. 
  
Moreover, every bicausal coupling induces a process coupling:
\begin{lemma}\label{lem:cpltoProcCpl}
For every $\fp X, \fp Y \in \FP$ and every $\pi \in\cplbc(\fp{X},\fp{Y})$ we have 
\[
(\Omega^\fp{X}\times\Omega^\fp{Y},\F^\fp{X}\otimes\F^\fp{Y},\pi,(\F^\fp{X}_t\otimes\F^\fp{Y}_t)_{t=1}^N,(X_t,Y_t)_{t=1}^N) \in\fpcpl(\fp X, \fp Y).
\] 
\end{lemma}
\begin{proof}
We need to show that $\fp X \simhk \overline{\fp X} := (\Omega^\fp{X}\times\Omega^\fp{Y},\F^\fp{X}\otimes\F^\fp{Y},\pi,(\F^\fp{X}_t\otimes\F^\fp{Y}_t)_{t=1}^N,X)$.  To that end, let $\sawu{X}$ be a self-aware lift of $\fp X$ and note that $\sawu{X}$ regarded as process on $\Omega^\fp{X}\times\Omega^\fp{Y}$ it is also self-aware lift of $\overline{\fp X}$. Indeed, we have for every $t \le N$, by first using the bicausality of $\pi$ in the form of Lemma~\ref{lem:causal_eqiv}\ref{it:causal CI1} and then the self-awareness of $\sawu{X}$ on the basis of $\fp X$
\[
\law_\pi(\sawu{X} | \F_t^{\fp X} \otimes \F_t^{\fp Y} ) = \law_{\P^{\fp X}}(\sawu{X}|\F_t^{\fp X}) = \law_{\P^{\fp X}}(\sawu{X}|\sawu{X}_{1:t}). \qedhere
\]
\end{proof}

We have seen in Remark~\ref{rem:embed} that for self-aware processes both definitions of bicausal couplings (i.e.\ Definition~\ref{def:bc_canonical} and Definition~\ref{def:causalcoup}) are compatible with each other. The notion of process couplings is also compatible with these definitions, as the following lemma details.

\begin{lemma}\label{lem:sabc} 
If  $\fp X, \fp Y \in \FP$ are self-aware, we have  
\begin{align}\label{eq:cplbceq}
\scplbc(\Law(X),\Law(Y))=\{\Law(\overline{X},\overline{Y}): \overline{\fp{X}}\marriage\overline{\fp{Y}}\in\fpcpl(\fp X, \fp Y)\}.
\end{align}
\end{lemma}
\begin{proof}
Write $\mu := \law(X)$ and $\nu := \law(Y)$. As both sides of \eqref{eq:cplbceq} are independent of the $\simhk$-representative, we can assume w.l.o.g.\ that $\fp X = \mathbb{S}^\mu$ and $\fp Y = \fp{S}^\nu$. First, let $\pi \in \scplbc(\mu,\nu)$ be given. By Lemma~\ref{lem:cpltoProcCpl} we have  
\[
(\X \times \Y, \B(\X \times \Y), \pi, \sigma(X_{1:t},Y_{1:t})_{t=1}^N, (X,Y)) \in \fpcpl(\fp X, \fp Y).
\]

To show the other inclusion in \eqref{eq:cplbceq}, let 
\[
\overline{\fp{X}}\marriage\overline{\fp{Y}} = (\Omega, \F, \P, (\F_t)_{t=1}^N, (\overline{X}, \overline{Y}) ) \in\fpcpl(\fp X, \fp Y)
\]
be given and write $\pi := \law(\overline{X}, \overline{Y})$. We need to show that $\pi \in \scplbc(\mu,\nu)$. Being self-aware is a property of the $\simhk$-class, so the process $\overline{\fp X}$ is self-aware as well, i.e.\ we have $\law_\P(\overline{X} |\F_t ) = \law_\P(\overline{X} | \overline{X}_{1:t} ) $ for every $t \le N$. As $\sigma(\overline{X}_{1:t}) \subset \sigma(\overline{X}_{1:t}, \overline{Y}_{1:t}) \subset \F_t$ this implies $\law_\P(\overline{X} | \overline{X}_{1:t}, \overline{Y}_{1:t}  ) = \law_\P(\overline{X} | \overline{X}_{1:t} )$, so $\pi$ is causal from $\overline{\fp X}$ to $\overline{\fp Y}$. Causality from $\overline{\fp Y}$ to $\overline{\fp X}$ can be proven in the same way. 
\end{proof}

Next, we discuss the transfer principle for filtered processes. Recall that usual transfer principle (see e.g.\ \cite[Theorem~5.10]{Ka97})  states that (under assumptions on the underlying spaces) 
\begin{itemize}
    \item Given random variables $(X,Y)$ on the same probability space and a random variable $\widetilde{X} \sim X$ on a further space, there exists a random variable $\widetilde{Y}$ on the same space as $\widetilde{X}$ such that $(\widetilde{X}, \widetilde{Y}) \sim(X,Y)$. 
\end{itemize}
The natural counterpart of this for filtered processes is:
\begin{itemize}
    \item Given a process coupling $\fp X \marriage \fp Y$ and a further filtered process $\widetilde{ \fp X} \simhk  \fp X$, there exists a filtered process $\widetilde{\fp Y}$ on the same stochastic basis as $\widetilde{\fp X}$ such that $\widetilde{ \fp X} \marriage  \widetilde{ \fp Y}  \simhk \fp X \marriage \fp Y$. 
\end{itemize}
The following theorem makes this statement precise including the assumptions on stochastic bases: 
\begin{theorem}[Transfer principle for filtered processes]\label{thm:transfer}
   Let $(\Omega, \F, \P, (\F_t)_{t=1}^N, (X_t,Y_t)_{t=1}^N)$ be a filtered process. If the filtered process  $\widetilde{\fp{X}}=(\widetilde{\Omega}, \widetilde{\F}, \widetilde{\P}, (\widetilde{\F}_t)_{t=1}^N, (\widetilde{X}_t)_{t=1}^N)$ admits independent randomization and $$\widetilde{\fp{X}}\simhk (\Omega, \F, \P, (\F_t)_{t=1}^N, (X_t)_{t=1}^N),$$ then there exists an adapted process $(\widetilde{Y}_t)_{t=1}^N$ on the stochastic basis of $\widetilde{\fp X}$  such that
    \begin{align*}
         (\Omega, \F, \P, (\F_t)_{t=1}^N, (X_t,Y_t)_{t=1}^N)\simhk(\widetilde{\Omega}, \widetilde{\F}, \widetilde{\P}, (\widetilde{\F}_t)_{t=1}^N, (\widetilde{X}_t,\widetilde{Y}_t)_{t=1}^N).
    \end{align*}
\end{theorem}
Before we start we the proof, note that the `Adjunction Theorem' \cite[Theorem 3.3]{Ho92} is essentially equivalent to our transfer principle. However, it is derived as a consequence of the `Amalgamation Theorem' \cite[Theorem 3.2]{Ho92}, which is wrong, even in the case of $N=1$ time-step  (see Appendix~\ref{sec:app} for a counterexample). Therefore, we give a proof of Theorem~\ref{thm:transfer} that is independent from the techniques in \cite{Ho92}. To that end, we need use the following result which is a representation of bicausal couplings by so-called biadapted maps on enlarged spaces. To that end, recall that a map $T:\X \times [0,1]^N \to\Y \times [0,1]^N$ is called adapted if it is of the form
\[
T(x,u)=(T_1(x_1,u_1), T_2(x_{1:2},u_{1:2}),\dots,T_N(x_{1:N},u_{1:N})). 
\]
A biadapted map is a bijection $T$ such that both $T$ and its inverse $T^{-1}$ are adapted. 

\begin{proposition}[\protect{\cite[Theorem 3.6]{BePaSc21c}}]\label{prop:biadaptedtm}
Let $\mu \in \mathcal{P}(\mathcal X)$ and  $\nu \in \mathcal{P}(\mathcal Y)$. If $\pi \in \scplbc(\mu,\nu)$, there exists a biadapted mapping  $T : \mathcal X \times [0,1]^N \to \mathcal{Y} \times [0,1]^N$ satisfying
\begin{enumerate}
	\item $T_\ast (\mu \otimes \lambda^N) =\nu \otimes \lambda^N$, or equivalently, $\widehat{\pi}:= (\id,T)_\ast(\mu \otimes \lambda^N) \in \scplbc(\mu \otimes \lambda^N,\nu \otimes \lambda^N)$,
	\item ${\pj_{\mathcal X\times \mathcal Y}}_\ast\widehat{\pi}=\pi$.
\end{enumerate}
\end{proposition}

Using this result, we can prove the transfer princple:
\begin{proof}[Proof of Theorem~\ref{thm:transfer}]
We assume without loss of generality that $(\Omega, \F, \P, (\F_t)_{t=1}^N)$ admits independent randomization, cf.\ Proposition~\ref{prop:eqivir}.

First note that there exist minimal self-aware lifts $A = (X, \iuc{A})$ of $\fp X$ and  $\widetilde{A} = (\widetilde{X}, \iuc{\widetilde A})$ of $\widetilde{\fp{X}}$ such that $\Law(A) = \Law(\widetilde{A})$, e.g.\ take the respective information processes. Further let $B = (X,Y, \iuc{B})$ be a self-aware lift of $\fp X \marriage \fp Y$. Denote the path-space of the lifts $A$ and $\widetilde A$ by $\A = \mathcal X \times \iuc{\A}$ and the path-space of the lift $B$ by $\B = \mathcal X \times \mathcal{Y} \times \iuc{\B}$. Moreover, write $\mu := \Law(A)$ and $ \nu := \Law(B)$. By Lemma~\ref{lem:sabc} we have $\pi := \Law(A,B) \in \scplbc(\mu,\nu)$. 

By Proposition~\ref{prop:biadaptedtm} there exists a biadapted map 
\[
T : \mathcal A \times [0,1]^N \to \mathcal{B} \times [0,1]^N
\]
such that 
\begin{align}\label{eq:prf:transf}
\widehat{\pi}:= (\id,T)_\ast(\mu \otimes \lambda^N) \in \scplbc(\mu \otimes \lambda^N,\nu \otimes \lambda^N) \quad  \text{and} \quad {\pj_{\mathcal A\times \mathcal B}}_\ast\widehat{\pi}=\pi.
\end{align}

Let $(\widetilde{U}_t)_{t=1}^N$ be an independent randomization for the self-aware lift $\widetilde{A}$. Then $T(\widetilde{A},\widetilde{U})$ is a self-aware process on the basis $(\widetilde{\Omega}, \widetilde{\F}, \widetilde{\P}, (\widetilde{\F}_t)_{t=1}^N)$. Indeed, $(\widetilde{A}_t, \widetilde{U}_t)_{t=1}^N$ is self-aware by Lemma~\ref{lem:randomization} and $\sigma(\widetilde{A},\widetilde{U}) = \sigma(T(\widetilde{A},\widetilde{U}))$ because $T$ is biadapted, hence
\[
\law(T(\widetilde{A},\widetilde{U}) | \widetilde{\F}_t) = T_\ast\law(\widetilde{A},\widetilde{U} | \widetilde{\F}_t) = T_\ast\law(\widetilde{A},\widetilde{U} | (\widetilde{A},\widetilde{U})_{1:t}  ) = \law(T(\widetilde{A},\widetilde{U}) | T(\widetilde{A},\widetilde{U})_{1:t}  ).
\]
Further note that $
T(\widetilde{A},\widetilde{U}) 
$ 
has paths in $\B \times [0,1]^N = \mathcal X \times \mathcal{Y} \times \iuc{\B} \times [0,1]^N$. We denote its respective components in this product space as 
$$
(\widetilde{X}', \widetilde{Y}, \iuc{\widetilde{B}}, \widetilde{V}) := T(\widetilde{A},\widetilde{U}), 
$$
more precisely, we set $ \widetilde{X}' := \pj_{\X} \circ T \circ (\widetilde{A}, \widetilde{U}) : \widetilde{\Omega} \to \X $ etc. As $\law(\widetilde{A}, \widetilde{U}) = \mu \otimes \lambda^N$, it follows from \eqref{eq:prf:transf} that
\[
\law_{\widetilde{\P}}(\widetilde{X}, \iuc{\widetilde A},  \widetilde{X}', \widetilde{Y}, \iuc{\widetilde{B}} ) = \pi = \law_\P(X, \iuc{A}, X, Y, \iuc{B}).
\]
Hence,  $\widetilde{X} = \widetilde{X}'$ $\widetilde{\P}$-a.s. Moreover,  $T(\widetilde{A}, \widetilde{U})$ is a self-aware lift of $(\widetilde{\Omega}, \widetilde{\F}, \widetilde{\P}, (\widetilde{\F}_t)_{t=1}^N, (\widetilde{X}_t,\widetilde{Y}_t)_{t=1}^N)$ and by \eqref{eq:prf:transf} we have $\law( T(\widetilde{A}, \widetilde{U}) ) = \nu \otimes \lambda^N$. If $V$ is an independent randomization for the self-aware lift $B$, then by Lemma~\ref{lem:randomization} the process $(B,V)$ is a self-aware lift of $(\Omega, \F, \P, (\F_t)_{t=1}^N, (X_t,Y_t)_{t=1}^N)$ that satsifes $\law(B,V)=\nu \otimes \lambda^N$. Hence, Theorem~\ref{thm:equivalence} implies that 
\[
(\Omega, \F, \P, (\F_t)_{t=1}^N, (X_t,Y_t)_{t=1}^N)\simhk(\widetilde{\Omega}, \widetilde{\F}, \widetilde{\P}, (\widetilde{\F}_t)_{t=1}^N, (\widetilde{X}_t,\widetilde{Y}_t)_{t=1}^N). \qedhere
\]
\end{proof}
\begin{remark}\label{rem:randomNecessary}
  Note that the assumption that $\widetilde{\fp X}$ satisfies independent randomization is necessary for the transfer principle to hold true. Suppose that the conclusion of the transfer principle holds true for a filtered process $\widetilde{\fp X}$. Let $\fp X \simhk \widetilde{\fp X}$ be a filtered process and $(Y_t)_{t=1}^N = (U_t)_{t=1}^N$ an independent randomization for $\fp X$. Then the process $(\widetilde{Y}_t)_{t=1}^N$ provided by Theorem~\ref{thm:transfer} is an independent randomization for $\widetilde{\fp X}$.  

  In particular, to obtain the transfer principle, it is not sufficient to postulate the existence of an independent uniform random variable at time $t=1$.
\end{remark}

\begin{remark} 
As already stressed above, an important reason to introduce adapted weak topologies and adapted Wasserstein distances has been to build a framework for  stochastic multistage optimization problems, see \cite{Al81, PfPi12, AcBaZa20, BaBaBeEd19a} among others. To this end, it desireable that once we identify processes $\fp X, \fp Y$, then also optimization problem written on these processes are  `equivalent' in the sense that they lead to the same value and that strategies can be translated from one setup to the other. 
An obstacle to this is that filtrations that carry the same information about the processes may still allow for different sets of candidates. 
Specifically if $\fp X$ allows for independent randomization while  $\fp Y$ carries only the intrinsic filtration, one obtains significantly different sets of strategies for optimization problems. 

One way to deal with this problem is to constrain admissible strategies w.r.t.\ the intrinsic filtration. Depending on the situation this may be undesirable, e.g.\ because randomization is  required to obtain compactness of the set of admissible strategies or because the natural optimizers are mixed strategies.

A perhaps preferable solution is to allow for general strategies, but to restrict to 
representatives  $\fp X\simhk \fp Y$ which allow for independent randomization. Then the transfer principle guarantees that strategies can indeed be transferred from one basis to the other, yielding equivalence of the respective optimization problems.   

Of course, for many problems (e.g.\ optimal stopping) it is not necessary to consider mixed strategies and it makes no difference which approach is chosen. 
\end{remark}

\section{Adapted Wasserstein distance}
The aim of this section is to give a representation of the adapted Wasserstein distance in terms of process couplings. 
	\begin{theorem}\label{thm:AWreform}
        For every $\fp X, \fp Y \in \FP_p$ we have 
		\begin{align}\label{eq:AWdefnew}
			\AW_p^p (\fp{X}, \fp{Y})= \inf_{\overline{\fp{X}}\marriage \overline{\fp{Y}}\in\fpcpl(\fp X, \fp Y) }  \E\left[ d(\overline{X},\overline{Y})^p\right].
		\end{align}
	The infimum on the right hand side is attained.

    If we further assume that $(\Omega^{\fp X}, \F^{\fp X})$ is standard Borel and $\fp X$ admits independent randomization, there is a process $(\overline{Y}_t)_{t=1}^N$ on the stochastic basis of $\fp X $ such that 
	\begin{align}\label{eq:sub}
		\E[d(X,\overline{Y})^p]^{\frac{1}{p}}=\AW_p(\fp{X}, \fp{Y})~\text{and}~\fp{Y}\simhk (\Omega, \F, \P, (\F_t)_{t=1}^N, \overline{Y} ) .
	\end{align}
\end{theorem}

\begin{proof}
First notice that both sides of \eqref{eq:AWdefnew} are invariant under passing to $\simhk$-equivalent processes, so we can assume w.l.o.g.\ that $\fp X$ and $\fp Y$ admit independent randomization. 

We first show that the left hand side in \eqref{eq:AWdefnew} is smaller than the right hand side. Indeed, whenever
$
\overline{\fp{X}} \marriage \overline{\fp{Y}}  = (\Omega,\F,\P, (\F_t)_{t=1}^N, (\overline{X}, \overline{Y})) \in\fpcpl(\fp X, \fp Y)
$
is a process coupling, we have $\fp X \simhk \overline{\fp X}$,  $\fp Y \simhk \overline{\fp Y}$ and  $(\id,\id)_\ast\P \in \scplbc(\overline{\fp X}, \overline{\fp Y})$. Hence, 
\[
\AW_p^p(\fp X, \fp Y ) = \AW_p^p(\overline{\fp X}, \overline{\fp Y}) \le \E_{(\id,\id)_\ast\P }[d(\overline{X}, \overline{Y})^p] = \E_\P[d(\overline{X}, \overline{Y})^p].
\]
To see the converse inequality, let $\pi \in \cplbc(\fp X, \fp Y)$ be an optimizer for $\AW(\fp X, \fp Y)$. Let $\overline{\fp{X}}\marriage \overline{\fp{Y}}$ be the process coupling induced by $\pi$ as in Lemma~\ref{lem:cpltoProcCpl}. Under this process coupling $\overline{\fp{X}}\marriage \overline{\fp{Y}}$ we have 
\[
\AW_p^p (\fp{X}, \fp{Y})=  \E\left[ d(\overline{X},\overline{Y})^p\right].
\]
Hence, we have equality in \eqref{eq:AWdefnew} and the infimum is attained by $\overline{\fp{X}}\marriage \overline{\fp{Y}}$. The second claim follows by applying Theorem~\ref{thm:transfer} to the process $\fp X$ and the process coupling $\overline{\fp{X}}\marriage \overline{\fp{Y}}$.
\end{proof}

\begin{remark} Different stability results of multistage optimization problems with respect to adapted Wasserstein distance  are established in  \cite{PfPi12, AcBaZa20, BaBaBeEd19a} among others. Typically the corresponding arguments use bicausal transport plans to transfer a candidate optimizers (e.g.\ a stopping time) from one space to another and to use bicausality properties to show that  admissible candidates is obtained in this way. The above representation of the adapted Wasserstein distance allows for a different view on this. 

Specifically, let us prove Lipschitz continuity of optimal stopping (see e.g.\ \cite{AcBaZa20, BaBaBeEd19b}), i.e.\ that if $f_t:\X_t\to \R $ is $L$-Lipschitz for all $t\leq N$, then the map
$$\fp{X} \to \sup\{ \E[f_\tau(X_1, \ldots, X_\tau)] : \tau \text{ st.~t.}\}$$
is $L$-Lipschitz w.r.t.\ $\AW_1$. 

Indeed, 
by Theorem \ref{thm:AWreform} filtered processes $\fp X, \fp Y$ can be assumed to be supported on the same stochastic basis and to satisfy $\AW_1(\fp X, \fp Y)= \E[d(X,Y)]$. Hence
\begin{multline*}
\big|\sup\{ \E[f_\tau(X_1, \ldots, X_\tau)] :  \tau \text{ st.~t.}\}- \sup\{ \E[f_\tau(Y_1, \ldots, Y_\tau)] : \tau \text{ st.~t.}\} \big| \\
\leq \sup\{ \E[|f_\tau(X_1, \ldots, X_\tau) - f_\tau(Y_1, \ldots, Y_\tau)|] : \tau \text{ st.~t.}\} \leq L\AW_1(\fp X, \fp Y).
\end{multline*}
The same type of argument can be used to establish other known stability results w.r.t.\ adapted Wasserstein distance. 
\end{remark}

A further application of Theorem~\ref{thm:AWreform} is the following easy representation of $\AW_p$-geodesics.

\begin{proposition} \label{prop:geodesic}
     Assume that $\X_t=\R^d$. Then $(\FFP_p,\AW_p)$ is a geodesic space. Moreover, the geodesic joining $\fp X$ and $\fp Y$ can be represented in the following way: Whenever 
     \[
     (\overline{\Omega}, \overline{\F}, \overline{\P}, (\overline{\F}_t)_{t=1}^N, (\overline{X}_t, \overline{Y}_t)_{t=1}^N)\in\fpcpl(\fp X, \fp Y)
     \]
     is a minimizer of \eqref{eq:AWdefnew}, a geodesic is given by the collection of processes
     \[
     \fp X^{\lambda} =  (\overline{\Omega}, \overline{\F}, \overline{\P}, (\overline{\F}_t)_{t=1}^N, ( (1-\lambda)\overline{X}_t + \lambda \overline{Y}_t)_{t=1}^N), \quad \lambda \in [0,1].
     \] 
 \end{proposition}
 \begin{proof}
 In order to show that $(\fp X^\lambda)_{\lambda \in [0,1]}$ is a geodesic, we need to check that 
 \[
 \AW_p(\fp{X}^{\lambda_1},\fp{X}^{\lambda_2})= |\lambda_1-\lambda_2| \AW_p(\fp{X},\fp{Y})
 \]
for all $\lambda_1, \lambda_2\in [0,1]$. Indeed, we have
  \begin{align*}
     \AW_p^p(\fp{X}^{\lambda_1},\fp{X}^{\lambda_2})&= \E_{\overline{\P}}\left[\sum_{t=1}^N\|((1-\lambda_1)\overline{X}_t +  \lambda_1 \overline{Y}_t)-((1-\lambda_2) \overline{X}_t +  \lambda_2 \overline{Y}_t)\|^p\right]\\
     &= |\lambda_1-\lambda_2|^p\E_{\overline{\P}}\left[ \sum_{t=1}^N\|\overline{X}_t - \overline{Y}_t\|^p\right]\\
     &=|\lambda_1-\lambda_2|^p \AW_p^p(\fp{X},\fp{Y}). \qedhere
 \end{align*}
 \end{proof}
Note that the statement of Proposition \ref{prop:geodesic} holds more generally for processes with values in an arbitrary geodesic space.


\section{Skorokhod Representation  for filtered processes}	
The aim of this section is to prove Skorohod-type results for adapted weak convergence, i.e.\ given a sequence $(\fp X^n)_n$ of filtered process converging to $\fp X$ in the adapted weak topology / adapted Wasserstein distance, there are $\simhk$-equivalent processes on the same stochastic base that converge almost surely / in $L_p$. 

Theorem~\ref{thm:AWreform} readily implies the following first Skorohod result for $\AW_p$ and $L_p$-convergence.

\begin{theorem}[$L_p$-Skorokhod representation]
	\label{thm:lpskorokhod}
	For $\fp X, \fp X^1, \fp X^2, \ldots \in \FP_p$ the following are equivalent:
	\begin{enumerate}[label=\emph{(\alph*)}]
		\item $\AW_p(\fp X^n, \fp X) \to 0$.
		\item There exist filtered processes $\overline{\fp X}\simhk\fp X$, $\overline{\fp  X}{}^n\simhk {\fp  X}^n$ on a common stochastic basis such that $\E[d(\overline{X}, \overline{ X}{}^n)^p]^\frac{1}{p}\to 0$.
	\end{enumerate}
	Moreover, the processes $\overline{\fp X}, \overline{\fp  X}{}^n $ in (b) can be chosen such that for every $n \in \N$
 \begin{align}\label{eq:Skorohod_Isometry} 
     \AW_p( {\fp X}, {\fp  X}^n)=\E[d(\overline{X}, \overline{ X}{}^n)^p]^\frac{1}{p}.
 \end{align}
\end{theorem}
\begin{proof}
 In order to show the implication from (a) to (b) and \eqref{eq:Skorohod_Isometry}, consider  $\fp X , \fp X_1, \fp X_2, \ldots  \in \FP_p$ such that $\AW_p(\fp X^n, \fp X) \to 0$. By Proposition~\ref{prop:eqivir}, there is a filtered process $\overline{\fp X} \simhk \fp X$  that allows for independent randomization and is based on a standard Borel space. For every $n \in \N$, an application of \eqref{eq:sub} to $\overline{\fp X}$ and $\overline{\fp X}{}^n$ yields a process $\overline{X}{}^n$ on the stochastic basis of $\overline{X}$ such that $\E[d(\overline{X}, \overline{X}{}^n)^p] = \AW^p_p(\fp X, \fp X{}^n)$. 

The implication from (b) to (a) is an easy consequence of Theorem~\ref{thm:AWreform} as $\overline{\fp X} \marriage \overline{\fp X}{}^n \in \fpcpl(\fp X,\fp X^n)$. 
\end{proof}

Next, we prove a more classical Skorokhod representation theorem with almost sure convergence instead of $L_p$-convergence. The idea of the proof is to inductively apply the classical Skorohod representation theorem for random variables with values in a Polish space to the information process. To avoid measurability problems in the induction step, we need a measurable version of the classical Skorokhod representation theorem. To prove this, we essentially repeat the construction  presented in \cite[Section~8.5]{Bo92}  while focusing on the specific aspects that are relevant for measurably and only giving an overview over the other aspects of the construction. For further details the reader is referred to \cite[Section~8.5]{Bo92}.
	
	
	\begin{proposition}\label{prop:SkBorel}
		Let $\cS$ be a Polish space. Then there is a Borel map 
		$$
		\Phi_{\cS}: \PP(\cS) \times [0,1] \to \cS
		$$ 
		satisfying
		\begin{enumerate}[label=\emph{(\roman*)}]
			\item $\Phi_{\cS}(\mu,\cdot)_\ast\lambda=\mu$. 
			\item If $\mu_n \to \mu$ weakly, then $\Phi_{\cS}(\mu_n, \cdot) \to \Phi_{\cS}(\mu,\cdot)$ $\lambda$-a.s. 
		\end{enumerate}
	\end{proposition}
	\begin{proof}
        \emph{Step 1: $\cS=C$.} We first cover the case, where $\cS$ is the classical Cantor set $C \subset [0,1]$. In this case, we choose $\Phi_{C}(\mu,\cdot)$ to be the quantile function $q_\mu$, i.e.\
		$$
		\Phi_{C}(\mu,u) := q_\mu(u) := \sup\{ x \in [0,1] : \mu([0,x)\cap C) \le u \}.
		$$

		It is easy to verify (i) and (ii). It remains to show that $\Phi_{C}$ is Borel. Since the Borel-$\sigma$-algebra on $C$ is the trace of the Borel-$\sigma$-algebra on $[0,1]$, it suffices to show that the preimages of sets of type $C \cap [0,a)$, where $a \in [0,1]$, are Borel. Indeed, we have
		$$
		\Phi_{\cS}^{-1} (C \cap [0,a))  =  \{ (\mu,u) \in  \PP(\cS) \times [0,1]: q_\mu(u) <a  \} = \{ (\mu,u) \in  \PP(\cS) \times [0,1]: \mu([0,a)\cap C)>u  \},
		$$
		which is clearly Borel.
		
		\emph{Step 2: $\cS = [0,1]^{\N}$.} \cite[Lemma~8.5.3]{Bo92} guarantees the existence of a continuous surjective map $F : C \to [0,1]^\N$ such that the map $\PP(C) \to \PP([0,1]^\N) : \mu \mapsto F_\ast\mu$ has a continuous right inverse $\Psi$. Then we set
		$$
		\Phi_{[0,1]^\N} = F \circ \Phi_{C} \circ (\Psi,\id_{[0,1]})
		$$
		and observe that it is Borel as concatenation of Borel maps. Property (i) is satisfied because we have 
		$$
		\Phi_{[0,1]^N}(\mu, \cdot) = F \circ \Phi_C(\Psi(\mu),\cdot) 
		$$
		and hence 
		$$
		\Phi_{[0,1]^N}(\mu, \cdot)_\ast\lambda= F_\ast\Phi_C(\Psi(\mu),\cdot)_\ast\lambda= F_\ast\Psi(\mu) = \mu.
		$$
		
		The continuity property (ii) is inherited from $\Phi_C$ because $F$ and $\Psi$ are continuous.
		
		\emph{Step 3: general $\cS$.} 
		Let $\iota : \cS \to [0,1]^\N$ be an embedding such that $\iota(\cS) \subseteq[0,1]^\N$ is $G_\delta$ and hence Borel. Let $j : [0,1]^\N \to \cS$ be defined via $j(\iota(s))=s$ and $j(x)=s_0$ if $x \notin \iota(\cS)$, where $s_0 \in \cS$ is an arbitrary fixed element. It is easy to see that $\Phi_{\cS}:=j \circ \Phi_{[0,1]^\N} \circ \PP(\iota)$, where $\PP(\iota)$ denotes the map $\PP(\cS) \to \PP([0,1]^\N): \mu \mapsto \iota_\ast\mu$, has the desired properties.
	\end{proof}

	\begin{lemma}\label{lem:ipdis}
		Let $\fp{X}=([0,1]^{N+1},\F,\lambda^{N+1},(\F_t)_{t=1}^{N+1},(X_t)_{t=1}^{N+1})$ where $\F_t=\sigma(\pj_{1:t})$ and $\F=\F^{N+1}$ be a filtered process  and denote $\fp{X}|_{\omega_1}=([0,1]^{N},\F,\lambda^{N},(\sigma(\pj_{2:t}))_{t=2}^{N+1},(X_t(\omega_1,\cdot)_{t=2}^{N+1}))$ for $\omega_1\in[0,1]$. Then we have for $t\leq N$
		\begin{align}
			\ip_{t+1}^{(N+1)}(\fp{X})(\omega_{1:t+1})=\ip_t^{(N)}(\fp{X}|_{\omega_1})(\omega_{2:t+1})\text{   for }\lambda\text{-almost all  }\omega_1\in[0,1].
		\end{align}
	\end{lemma}
	\begin{proof}
		We prove this by backward induction on $t$. For $t=N$ it holds that
  \[
  ip_{N+1}^{(N+1)}(\fp{X})(\omega_{1:N+1})=X_{N+1}(\omega_{1:N+1}) 
			=\ip_N^{(N)}(\fp{X}|_{\omega_1})(\omega_{2:N+1}).
  \]
		We show the statement holds for $t-1$, assuming it holds true for $t$ in the second equality:
		\begin{align*}
			\ip_{t-1+1}^{(N+1)}(\fp{X})(\omega_{1:t})&=(X_t(\omega_{1:t}),\Law(\ip_{t+1}(\fp{X})\vert\F_t)(\omega_{1:t}))\\
			&=(X_t(\omega_{1:t}),\Law(\ip_{t}(\fp{X}|_{\omega_1})\vert\F_t)(\omega_{2:t}))\\
			&=\ip_{t-1}^{(N)}(\fp{X}|_{\omega_1})(\omega_{2:t}). 
		\end{align*}
		So we are able to conclude that the statement applies to all $t\leq N$.
	\end{proof}

Let $\FFP_N$ denote the space of $\mathcal S$-valued filtered processes with $N$ time steps\footnote{I.e.\  $X_t$ takes values in $\mathcal{S}$ for every $1\le t \le N$. The assumption that all random variables take values in the same Polish space $\mathcal{S}$ and not in different image spaces $\X_1, \dots, \X_N$ eases the notion in the proof of the following theorem and is no loss of generality as we can take $\mathcal{S}$ to be the disjoint union of these spaces $\X_1, \dots, \X_N$.
}. The next theorem is a consequence of iterative applications of Proposition~\ref{prop:SkBorel} to the information process.
	\begin{theorem}
		\label{thm:asskorokhodmap}
		For every $N\in \N$ there is a Borel map
		$$
		\Phi^{(N)}: \PP(\Z^{(N)}_1) \times [0,1]^N  \to \mathcal{S}^N
		$$
		such that for every $\mu \in \PP(\Z^{(N)}_1)$ 
		$$
		\fp X^\mu := ([0,1]^N, \B_{[0,1]^N}, \lambda^N, (\pj_{1:t})_{t=1}^N, \Phi^{(N)}(\mu, \cdot)) 
		$$
		is a filtered process satisfying $\Law(\ip^{(N)}_1(\fp X^\mu))=\mu$. Moreover, if $\mu_n \to \mu$ weakly, then  $\Phi^{(N)}(\mu_n,\cdot) \to \Phi^{(N)}(\mu, \cdot)$ $\lambda^N$-a.s.
	\end{theorem}
	\begin{proof}
		The statement will be proven by induction on $N$. In the case $N=1$, we can set $\Phi^{(1)} := \Phi_{\cS}$, where $\Phi_{\cS}$ is the map introduced in Proposition~\ref{prop:SkBorel}. Since $\ip_1^{(1)}(\fp X^\mu) = \Phi^{(N)}(\mu,\cdot) = \Phi_{\cS}(\mu,\cdot)$, it is immediate that $\Phi^{(1)}$ has the required properties.
		
		For the inductive step, assume that there is a Borel map
		$$
		\Phi^{(N)}: \PP(\Z^{(N)}_1) \times [0,1]^N  \to \mathcal{S}^N
		$$
		with the desired properties. We aim to construct the map $\Phi^{(N+1)}$. Proposition~\ref{prop:SkBorel} applied to the space $\Z^{(N+1)}$ yields a Borel map 
		$$
		\Phi_{\Z_1^{(N+1)}} : \PP(\Z^{(N+1)})  \times [0,1] \to\Z^{(N+1)}
		$$
		that satisfies the conditions (i) and (ii) stated in this proposition.  We can define $\Phi^{(N+1)}$ as
		$$
		\Phi^{(N+1)}:= (\id_{\cS} \times \Phi^{(N)}) \circ ( \Phi_{\Z^{(N+1)}} \times \id_{[0,1]^N}).
		$$
		Clearly, $\Phi^{(N+1)}$ is Borel as a concatenation of Borel maps.
		The following display explains in more detail how $\Phi^{(N+1)}$ acts:
		\begin{alignat*}{1}
			&\PP(\Z_1^{(N+1)}) \times [0,1]^{N+1}  \!\!\xrightarrow{ \Phi_{\Z_1^{(N+1)}} \times \id_{[0,1]^N}\!} \cS \times \PP(\Z_1^{(N)}) \times [0,1]^N \xrightarrow{\quad\;\id_{\cS} \times \Phi^{(N)}\quad\;\; }\; \cS \times \cS^N\\
			&\qquad\;\;(\mu, \omega_{1:N+1}) \xmapsto{\qquad\qquad\qquad\;}  (Z^-(\omega_1),Z^+(\omega_1),\omega_{2:N+1}) \xmapsto{\qquad\!}(Z^-(\omega_1), \Phi^{(N)}(Z^+(\omega_1),\omega_{2:N+1} )).
		\end{alignat*}
		Here, $Z$ is the $\Z_1^{(N+1)}$-valued random variable $\Phi_{\Z^{(N+1)}_1}(\mu,\cdot)$ and 
		we write its components as $Z=(Z^-,Z^+) \in \cS \times \Z_1^{(N)}$. 
		
		Let $\mu \in \PP(\Z^{(N+1)}_1)$. We need to show  that  the filtered process  
		$$
		\fp X^\mu := ([0,1]^{N+1}, \B_{[0,1]^{N+1}}, \lambda^{N+1}, \sigma(\pj_{1:t})_{t=1}^{N+1}, \Phi^{(N+1)}(\mu, \cdot)) 
		$$
		satisfies $\Law(\ip^{(N+1)}_1(\fp X^\mu))=\mu$. By the induction hypothesis, for almost all $\omega_1$, the filtered process
		$$
		\fp X^{Z^+(\omega_1)} := ([0,1]^N, \B_{[0,1]^N}, \lambda^N, \sigma(\pj_{1:t})_{t=1}^N, \Phi^{(N)}(Z^+(\omega_1), \cdot)) 
		$$
		satisfies $\Law(\ip_1^{(N)}( \fp X^{Z^+(\omega_1)} )  ) = Z^+(\omega_1)$. 

When setting $\fp X = \fp X^{\mu}$ and $\fp X|_{\omega_1} \!= \fp X^{Z^+(\omega_1)}$, we are precisely in the setting of Lemma~\ref{lem:ipdis} and can therefore conclude that $\ip_2^{(N+1)}(\fp X^\mu) = \ip_1^{(N)}(\fp X^{Z^+(\cdot)})$. Using these observations, we see that
		\begin{align*}
			\ip_1^{(N+1)}(\fp X^\mu) &= (\Phi_1^{(N+1)}(\mu,\cdot ), \Law(\ip_2^{(N+1)}(\fp X^\mu) |\F_1^{(N+1)} ))  =(Z^-, \Law(\ip_1^{(N)}(\fp X^{Z^+}) |\F_0^{(N)} ))\\
			&=(Z^-,\Law(\ip_1^{(N)}(\fp X^{Z^+}))) = (Z^-,Z^+)=Z,
		\end{align*}
		and hence $\Law(\ip^{(N+1)}(\fp X^\mu)) = \Law(Z) = \mu$. 
		
		Finally, assume that $\mu^n \to \mu$.  Then we have $(Z_n^-,Z_n^+) \to (Z^-,Z^+)$ $\lambda$-a.s by property (ii) of $\Phi_{\Z_1^{(N+1)}}$. In particular, one a $\lambda$-full set, the sequence of measures $Z^{+}_n(\omega_1) \to Z^+(\omega_1)$ weakly. Hence,  $\Phi^{(N)}(Z^+_n(\omega_1),\cdot  ) \to \Phi^{(N)}(Z^+(\omega_1),\cdot  )$ almost surely by the inductive hypothesis. All in all, this yields $\Phi^{(N+1)}(\mu_,\cdot) \to \Phi^{(N+1)}(\mu_,\cdot)$ $\lambda^{N+1}$-a.s.
	\end{proof}

One can interpret the map constructed in Theorem~\ref{thm:asskorokhodmap} as a map that assigns every filtered process a canonical representative in such a way that convergence in the adapted weak topology is equivalent to $ \lambda^N$-a.s.\ convergence of the representatives.

\begin{corollary}\label{cor:skorohod}
	There exists a mapping on $\FFP$ that assigns to every filtered process $\fp X$ random variables $T_t^\fp{X}:[0,1]^t\rightarrow\X_{t}$ for $1\leq t\leq N$ such that the filtered  process
	\begin{align*}
	\fp{X} \simhk ([0,1]^N, \B_{[0,1]^N}, \lambda^N, \sigma(\pj_{1:t})_{t=1}^N, (T_t^\fp{X})_{t=1}^N) 
	\end{align*}
	and the following assertions hold true: 
 \begin{enumerate} [label=\emph{(\roman*)}]
     \item For all $\fp X^n, \fp X \in \FP$ we have $\fp X^n \to \fp X$ in the adapted weak topology if and only if $(T_t^{\fp X^n})_{t=1}^N \to (T_t^{\fp X})_{t=1}^N$ $\lambda$-a.s.
     \item For all $\fp X^n, \fp X \in \FP_p$ we have $\AW_p(\fp X^n, \fp X) \to 0$ if and only if $(T_t^{\fp X^n})_{t=1}^N \to (T_t^{\fp X})_{t=1}^N$ in $L^p([0,1]^N, \lambda^N)$.
\end{enumerate}
\end{corollary}
\begin{proof}
We set $(T_t^{\fp X})_{t=1}^N := \Phi^{(N)}(\law(\ip_1(\fp X)))$, where $\Phi^{(N)}$ is the map defined in Theorem~\ref{thm:asskorokhodmap}. Then claim (1) is immediate. In order to prove (2), note that $\AW_p(\fp X^n, \fp X) \to 0$ implies that $\law((T_t^{\fp X^n})_{t=1}^N) \to  \law((T_t^{\fp X})_{t=1}^N)$ in $\W_p$. Hence, the $p$-th moments of these random variables are uniformly integrable, so pointwise convergence implies $L_p$-convergence. 
\end{proof}

A further consequence of this result is
\begin{corollary}
Let $\X$ be a separable Banach space. Then the map $\fp X \mapsto T^{\fp X}$ introduced in Corollary~\ref{cor:skorohod} is a homeomorphism between $\FFP$ and a closed subset of $L_0([0,1]^N,\lambda^N;\X)$.
\end{corollary}
\begin{proof}
It suffices to observe that the image of the map $\fp X \mapsto T^{\fp X}$ is closed. To that end, suppose that $(T^{\fp X^n})_n$ converges to some $Y \in L_0([0,1]^N,\lambda^N;\X)$. Then consider the filtered process $\fp{X} := ([0,1]^N, \B_{[0,1]^N}, \lambda^N, \sigma(\pj_{1:t})_{t=1}^N, Y)$. An application of Theorem~\ref{thm:AWreform} (with a truncated distance) yields that 
 $\fp X^n \to \fp X$ in the adapted weak topology. Hence, by Corollary~\ref{cor:skorohod}, we have $T^{\fp X^n} \to T^{\fp X}$ in $L_0$ and hence by uniqueness of the limit $Y=T^{\fp X}$.
\end{proof}

We conclude this section with a discussion of the Skorohod results presented so far. 
\begin{remark}
The arguments given in the proof of Theorem~\ref{thm:lpskorokhod} prove a slightly more general statement: Fix $\fp X \in \FFP_p$ and choose a $\simhk$-representative of $\fp X$ denoted as
\[
(\Omega,\F,\P,(\F_t)_{t=1}^N, X)
\]
that admits independent randomization. Then for every further $\fp Y \in \FFP_p$, there exists a process $\overline{Y}$ on the stochastic basis $(\Omega,\F,\P,(\F_t)_{t=1}^N)$ such that
\[
\fp Y \simhk (\Omega,\F,\P,(\F_t)_{t=1}^N, \overline{Y}) \quad \text{and}  \quad \AW_p^p(\fp X, \fp Y) = \E [d(X,\overline{Y})^p].
\]

However, given two filtered processes $\fp Y^1, \fp Y^2 \in \FFP_p$ and processes $\overline{Y}{}^1,\overline{Y}{}^2$ on $(\Omega,\F,\P,(\F_t)_{t=1}^N)$ such that for $i \in \{1,2\}$
\[
\fp Y^i \simhk (\Omega,\F,\P,(\F_t)_{t=1}^N, \overline{Y}{}^i) \quad \text{and}  \quad \AW_p^p(\fp X, \fp Y) = \E [d(X,\overline{Y}{}^i)^p]
\]
we generally do \emph{not} have $\AW_p^p(\fp Y^1, \fp Y^2) = \E [d(\overline{Y}{}^1,\overline{Y}{}^2)^p]$. This behaviour is not surprising and it amounts to the fact that in optimal transport the gluing of optimal transport plans does in general not yield an optimal transport plan.

In the case of filtered processes with values in the Euclidean  space $\R^d$ and $p=2$, there is a geometric interpretation of these results. For this, we think of $\fp X$ as a reference point in the geodesic space $(\FFP_2,\AW_2)$. Representing the distance from $\fp X$ to any further point $\fp Y \in \FFP_p$ in the Hilbert space $L^2((\Omega,\P); \R^{dN})$ as
\[
\AW_p(\fp X, \fp Y) = \| X- \overline{Y} \|_{ L^2((\Omega,\P); \R^{dN})}
\]
is reminiscent to choosing a geodesic normal coordinate system with base point $\fp X$. The fact that for the three points $\fp X, \fp Y^1, \fp Y^2$ it is not possible to represent all pairwise distances on the Hilbert space $L^2((\Omega,\P); \R^{dN})$ corresponds to the fact that space $(\FFP_2,\AW_2)$ is not flat (except for the special case $d=N=1$). In fact, $(\FFP_2,\AW_2)$ has positive curvature in the sense of Alexandrov. This is because $(\FFP_2,\AW_2)$ is isometric to a nested Wasserstein space (see Theorem~\ref{thm:isometry}) and the fact that the Wasserstein space over a positively curved space is positively curved (see \cite[Theorem~2.20]{AGguide}). 
\end{remark}

The previous remark has shown that in Theorem~\ref{thm:lpskorokhod} the isometric property \eqref{eq:Skorohod_Isometry} between $\fp X$ and $\fp X^n$ and can in general not be extended to all pairwise distances between the processes $\fp X, \fp X_1, \fp X_2, \dots$ at the same time. The next proposition shows that it is also not poosible to enforce the isometric property \eqref{eq:Skorohod_Isometry} and almost sure convergence at the same time. Note that this is a again not a new phenomenon of adapted transport, but already the case for the classical Wasserstein distance:
\begin{proposition} 
For every $p  \in [1, \infty)$ the following assertions are true:
\begin{enumerate}[label=\emph{(\roman*)}]
    \item There exist $\mu, \mu^1, \mu^2,\ldots  \in \mathcal{P}_p(\R^2)$ such that $\W_p(\mu^n,\mu) \to 0$ with the following property: There is no probability space $(\Omega,\F,\P)$ with random variables $X^n \sim \mu^n$ and $X \sim \mu$ such that both $\E_\P[|X^n - X|^p] = \W^p_p(\mu^n, \mu)$ and $X^n \to X $ $\P$-a.s. 
    \item There exist real-valued filtered processes $\fp X, \fp X^1, \fp X^2, \ldots$ in $N=2$ time periods with the following property: There is no stochastic basis $(\Omega,\F,\P, (\F_t)_{t=1}^2)$ with processes $(X_t)_{t=1}^2$ and $(X_t^n)_{t=1}^2$ such that
    \[
    (\Omega,\F,\P, (\F_t)_{t=1}^2,X^n) \simhk \fp X^n, \qquad (\Omega,\F,\P, (\F_t)_{t=1}^2,X) \simhk \fp X,
    \]
    $\E_\P[|X^n - X|^p] = \AW^p_p(\fp X^n, \fp X)$ and $X^n \to X $ $\P$-a.s.
\end{enumerate}
\end{proposition}
\begin{proof}
Set $a_n=\sum_{k=1}^{n}\frac{1}{k}$ and  let $A_n$ be the segment between the points $a_{n-1} \mod 1$ and  $a_{n} \mod 1$. It is easy to see that $1_{A_n} \to 0$ in $L^p([0,1],\lambda)$, but not $\lambda$-a.s. We aim to use this property to define $\mu^n,\mu \in \mathcal{P}(\R^2)$ with the desired properties. To that end, set $\mu^n := (\id, 1_{A_n})_\ast \lambda$, where $\lambda$ is the Lebesgue measure on $[0,1]$ and $\id: [0,1] \to [0,1]$ is the identity, and further set $\mu := (\id,0)_\ast \lambda$. 
In order to prove (a),  we show that $\W_p(\mu,\mu^n) = 1/n$ and that the unique optimal coupling is given by $\pi^n := g^n_\ast\lambda$, where $g^n(x)=(x,0,x,1_{A_n}(x))$. It is easy to show that the cost of $\pi^n$ is $1/n$. If $X=(X_1,X_2)$ is a random variable  with $\law(X)=\mu$ and $X^n=(X^n_1,X^n_2)$ is a random variable with $\law(X^n)=\mu^n$ such that the coupling $(X,X^n)$ has cost at most $1/n$, we have 
\begin{align*}
n^{-p} \ge \E[ |X^n - X|^p ] = \E[ |X_1^n - X_1|^p ] + \E[ |X_2^n - X_2|^p ] =  \E[ |X_1^n - X_1|^p ] + n^{-p}.    
\end{align*}
Therefore, $X_1^n = X_1$ a.s.\ and hence $X_2^n=1_{A_n}(X^n_1)=1_{A_n}(X_1)$ a.s. This shows that $\law(X,X^n)=\pi$, hence $\pi^n$ is the unique optimal coupling.

Next, we observe that whenever $X,X^1,X^2,\dots$ are random variables on a probability space $(\Omega, \F, \P)$ such that $\law(X,X^n)=\pi^n$ for every $n \in \N$, then $X^n$ does not converge to $X$ $\P$-a.s. Indeed, by the definition of $\pi^n$, we have that $X_2^n=1_{A_n}(X_1)$ a.s. Since $(1_{A_n})_n$ does not converge $\lambda$-a.s. and $\law(X_1)=\lambda$, this yields that $(X_2^n)_n$ does not converge $\P$-a.s. This proves (i).

In order to show (ii), we consider the self-aware filtered processes $\fp X^n := \mathbb{S}^{\mu^n}$, $\fp X := \mathbb{S}^\mu$. As the coupling $\pi^n$ defined above is bicausal, Remark~\ref{rem:embed} implies that 
\[
\AW(\fp X^n , \fp X ) = \AW(\mu^n, \mu) = \W(\mu^n, \mu) = 1/n. 
\]
Let $(\Omega,\F,\P, (\F_t)_{t=1}^2)$ be a stochastic basis and let  $ (X_t^n)_{t=1}^2,  (X_t)_{t=1}^2$ be processes on this basis that satisfy
    \[
    (\Omega,\F,\P, (\F_t)_{t=1}^2,X^n) \simhk \fp X^n, \,\, (\Omega,\F,\P, (\F_t)_{t=1}^2,X) \simhk \fp X, \,\, \text{and } \E_\P[|X^n - X|^p] = \AW^p_p(\fp X^n, \fp X).
    \]
Then, by the same reasoning as above, we have $\law(X,X^n)=\pi^n$ and hence $X^n$ does not converge to $X$ $\P$-a.s.  
\end{proof}

\appendix
\section{}\label{sec:app}
The purpose of this appendix is to present a counterexample to the amalgamation theorem  \cite[Theorem~3.2]{Ho92}. To that end, we provide a counterexample to \cite[Lemma 3.1]{Ho92}, which is not only used in the proof of \cite[Theorem~3.2]{Ho92}, but also a special case of it, namely the case of one discrete time step. In order to state \cite[Lemma 3.1]{Ho92}, we need the notion of measure-algebra isomorphisms.
\begin{definition}[$\sigma$-algebra isomorphisms]
  Let $(\Omega, \F, \P)$ and $(\widetilde \Omega, \G, \Q)$ be probability spaces with nullsets $\mathcal{N}_\P$ and $\mathcal{N}_\Q$ respectively. The map $h:\F/\mathcal{N}_\P\rightarrow \G/\mathcal{N}_\Q$ is a measure-algebra isomorphism if for all $F_1,F_2\in\F$ with $G_i/\mathcal{N}_\Q=h(F_i/\mathcal{N}_\P)$ the following conditions hold:
  \begin{enumerate}
      \item $\P(F_1)=\Q(G_1)$ 
      \item If $F_1\subseteq F_2$ then $G_1\subseteq G_2$ up to $\Q$-nullsets
  \end{enumerate}
\end{definition}

\begin{lemma}[{\cite[Lemma 3.1]{Ho92}}]
Let $(\Omega^j, \G^j, \P^j)$, $j\in J$, be probability spaces, and let $\F^j\subseteq \G^j$, $j\in J$, be sub-$\sigma$-fields, with measure-algebra isomorphisms $h^j:(\F, \P)\rightarrow(\F^j, \P^j)$, where $(\F,\P) = (\F^{j_0}, \P^{j_0})$ for some $j_0\in J$, and $h^{j_0}$ is the identity. Then there is a probability measure $\Q$ on $(\prod_{j \in J} \Omega^j, \bigotimes_{j \in J} \G^j)$ such that for every $G^j \in \G^j$, $j \in J$ 
\begin{align}
\label{eq:prodform}
\Q\Big( \prod_{j \in J} G^j \Big)=\int\prod_{j \in J} (h^j)^{-1} (\P^j(G^j|\F^j))\emph{d}\P.
\end{align}
If $\pi^j$, $j\in J$ are the coordinate projections $\omega\mapsto\omega_j$, then $(\pi^j)^{-1}(\F^j)$, $j\in J$ are identical modulo nullsets, and $(\pi^j)^{-1}(\G^j)$, $j\in J$ are conditional independent given $(\pi^j)^{-1}(\F^j)$.
\end{lemma}

There is a counterexample to this lemma, even for $J=\{1,2\}$.
\begin{example}
\label{ex:amalgamation}
Let $J= \{1,2\}$ and $j_0=1$. Further, let $\Omega^1=B$ and $\Omega^2=B^c$, where $B \subset [0,1]$ is a Bernstein set\footnote{A set $B \subset [0,1]$ is called Bernstein set, if $B \cap C \neq \emptyset$ and $B^c \cap C \neq \emptyset$ for every uncountable closed set $C \subset [0,1]$. Bernstein sets exist, see e.g.\ \cite[Theorem 5.3]{Ox80}. Clearly, $B$ is Bernstein if and only if $B^c$ is Bernstein. Every measurable subset of a Bernstein set is a nullset, see e.g.\ \cite[Theorem 5.4]{Ox80}.}. %
Let $\F^1 = \G^1 =\{A\cap B: F\in\B([0,1])\}$ and $\F^2 = \G^2 =\{A\cap B^c: A\in\B([0,1])\}$ be the respective trace-$\sigma$-algebras of the Borel-$\sigma$-algebra on $[0,1].$ Further, define probability measures as $\P^1(A \cap B) = \lambda(A)$, $A \in \B([0,1])$ and $\P^2(A \cap B^c) = \lambda(A)$,  $A \in \B([0,1])$. Further set $(\Omega,\G,\P) = (\Omega^1,\G^1,\P^1)$ and $\F = \F^1$.

First, we show that $\P^1$ is well-defined (the case of $\P^2$ is analog). Indeed, if $A \cap B = A' \cap B$ for $A, A' \in \B([0,1])$, then we have $A \Delta A' \subset B^c$. As every measurable subset of a Bernstein set is a nullset, we find $\lambda( A \Delta A' ) = 0$, hence $\lambda(A) = \lambda(A')$. Moreover, it is easy to check that $\P^1$ is $\sigma$-additive and satisfies $\P^1(\emptyset)=0, \P^1(B)=1$.

Let $h^1$ be the identity and let $h^2:(\F^1,\P^1)\rightarrow(\F^2,\P^2) $ be defined by $h^2(A\cap B)= A\cap B^c$. By the same reasoning as above, it is easy to see that $h_2$ is indeed a measure algebra isomorphism.

Assume now that $\Q$ is a probability measure on $(\Omega^1 \times \Omega^2, \G^1 \otimes \G^2)$ that satisfies \eqref{eq:prodform}. Let $A^1, A^2 \in \B([0,1])$ and consider the sets $G_1 = A^1 \cap B \in \F^1$ and $G_2 = A^2 \cap B^c \in \F^2$. As $G^j$ is $\F^j$-measurable for $j=1,2$ and $(h^2)^{-1}(G^2)=A^2 \cap B$, we  have
\begin{align}\label{eq:ex:1}
  \begin{split}
    \Q(G^1 \times G^2) &= \int\! \P^1(G^1|\F^1) \!\cdot\!(h^2)^{-1}(\P^2(G^2|\F^2)) d\P \\
    &= \int 1_{G^1} 1_{ (h^2)^{-1} ( G^2)}  d\P 
    = \P( A^1 \cap A^2 \cap B ) = \lambda(A^1 \cap A^2). 
  \end{split}
\end{align}
We aim to use this property of $\Q$ to derive a contradiction. To this end, for $k \in \N$ and $i \in \{0, \dots k-1\}$, consider the set
\[
G_{i,k} :=G^1_{i,k}\times G^2_{i,k} =\left(\left[\frac{i}{k},\frac{i+1}{k}\right)\cap B\right)\times \left(\left[\frac{i}{k},\frac{i+1}{k}\right)\cap B^c\right).
\]
Further set $H_k = \bigcup_{i=0}^{k-1}G_{i,k}$. Using \eqref{eq:ex:1}, we find for every $k \in \N$
\begin{align*}
\Q(H_k)=\sum_{i=0}^{k-1}\Q(G_{i,k})= \sum_{i=0}^{k-1} \lambda \left( \left[\frac{i}{k},\frac{i+1}{k}\right) \right) = 1. 
\end{align*}
Therefore, we have 
\begin{align}\label{eq:ex:2}
    \Q\left(\bigcap_{k \in \N} H_k\right) =1.
\end{align}
On the other hand, it is easy to see that 
\[
H_k \subset \{ (x,y) \in  B \times B^c : |x-y|<1/k \}.
\]
Hence, $\bigcap_{k \in \N} H_k = \emptyset$, which is a contradiction to \eqref{eq:ex:2}.
\end{example}

\bibliographystyle{abbrv} 
\bibliography{joint_biblio, biblio_new}

\begin{thebibliography}{10}

\bibitem{AcBaZa20}
B.~Acciaio, J.~Backhoff-Veraguas, and A.~Zalashko.
\newblock Causal optimal transport and its links to enlargement of filtrations and continuous-time stochastic optimization.
\newblock {\em Stoch.~Proc.~Appl.}, 130(5):2918--2953, 2020.

\bibitem{AcBePa20}
B.~Acciaio, M.~Beiglb\"{o}ck, and G.~Pammer.
\newblock Weak transport for non-convex costs and model-independence in a fixed-income market.
\newblock {\em Math. Finance}, 31(4):1423--1453, 2021.

\bibitem{AcHo22}
B.~Acciaio and S.~Hou.
\newblock Convergence of adapted empirical measures on $\mathbb{R}^{d}$.
\newblock {\em arXiv preprint arXiv:2211.10162}, 2023.

\bibitem{AcKrPa23}
B.~Acciaio, A.~Kratsios, and G.~Pammer.
\newblock Designing universal causal deep learning models: The geometric (hyper) transformer.
\newblock {\em Mathematical Finance}, 2023.

\bibitem{XuLiMuAc20}
B.~Acciaio, M.~Munn, L.~Wenliang, and T.~Xu.
\newblock Cot-gan: Generating sequential data via causal optimal transport.
\newblock In {\em Advances in Neural Information Processing Systems}, volume~33, pages 8798--8809, 2020.

\bibitem{AkGaKi23}
S.~Akbari, L.~Ganassali, and N.~Kiyavash.
\newblock Learning causal graphs via monotone triangular transport maps.
\newblock {\em arXiv:2305.18210}, 2023.

\bibitem{Al81}
D.~J. Aldous.
\newblock Weak convergence and general theory of processes.
\newblock Unpublished monograph: Department of Statistics, University of California, Berkeley, July 1981.

\bibitem{AGguide}
L.~Ambrosio and N.~Gigli.
\newblock A user's guide of optimal transport.
\newblock {\em preprint}, 2011.

\bibitem{AmGiSa08}
L.~Ambrosio, N.~Gigli, and G.~Savar{\'e}.
\newblock {\em Gradient flows in metric spaces and in the space of probability measures}.
\newblock Lectures in Mathematics ETH Z\"urich. Birkh\"auser Verlag, Basel, second edition, 2008.

\bibitem{BaBaBeEd19a}
J.~Backhoff-Veraguas, D.~Bartl, M.~Beiglb\"{o}ck, and M.~Eder.
\newblock Adapted {W}asserstein distances and stability in mathematical finance.
\newblock {\em Finance Stoch.}, 24(3):601--632, 2020.

\bibitem{BaBaBeEd19b}
J.~Backhoff-Veraguas, D.~Bartl, M.~Beiglb\"{o}ck, and M.~Eder.
\newblock All adapted topologies are equal.
\newblock {\em Probab.~Theory Relat.~Fields}, 178(3-4):1125--1172, 2020.

\bibitem{BaBaBeWi20}
J.~Backhoff-Veraguas, D.~Bartl, M.~Beiglb{\"o}ck, and J.~Wiesel.
\newblock Estimating processes in adapted {W}asserstein distance.
\newblock {\em Ann. Appl. Probab.}, 32(1):529--550, 2022.

\bibitem{BaBeLiZa17}
J.~Backhoff-Veraguas, M.~Beiglb\"{o}ck, Y.~Lin, and A.~Zalashko.
\newblock Causal transport in discrete time and applications.
\newblock {\em SIAM J.~Optim.}, 27(4):2528--2562, 2017.

\bibitem{BaBePa19}
J.~Backhoff-Veraguas, M.~Beiglb{\"o}ck, and G.~Pammer.
\newblock Weak monotone rearrangement on the line.
\newblock {\em {Electronic Communications in Probability}}, 25, 2020.

\bibitem{BaBePa21}
D.~Bartl, M.~Beiglb{\"o}ck, and G.~Pammer.
\newblock The {W}asserstein space of stochastic processes.
\newblock {\em arXiv:2104.14245}, 2023.

\bibitem{BaWi23}
D.~Bartl and J.~Wiesel.
\newblock Sensitivity of multiperiod optimization problems with respect to the adapted wasserstein distance.
\newblock {\em SIAM J.~Financ.~Math.}, 14(2):704--720, 2023.

\bibitem{BaDoDo20}
E.~Bayraktar, L.~Dolinskyi, and Y.~Dolinsky.
\newblock Extended weak convergence and utility maximisation with proportional transaction costs.
\newblock {\em Finance Stoch.}, 24(4):1013--1034, 2020.

\bibitem{BaHa23}
E.~Bayraktar and B.~Han.
\newblock Fitted value iteration methods for bicausal optimal transport.
\newblock {\em arXiv:2306.12658}, 2023.

\bibitem{BeJoMaPa21a}
M.~Beiglb\"{o}ck, B.~Jourdain, W.~Margheriti, and G.~Pammer.
\newblock Approximation of martingale couplings on the line in the adapted weak topology.
\newblock {\em Probab. Theory Related Fields}, 183(1-2):359--413, 2022.

\bibitem{BePaPo21}
M.~{Beiglb\"ock}, G.~{Pammer}, and A.~Posch.
\newblock The {K}nothe--{R}osenblatt distance and its induced topology.
\newblock {\em ArXiv e-prints}, 2023.

\bibitem{BePaSc21c}
M.~Beiglb{\"o}ck, G.~Pammer, and S.~Schrott.
\newblock Denseness of biadapted {M}onge mappings.
\newblock {\em arXiv:2210.15554}, Oct. 2022.

\bibitem{BiTa19}
J.~Bion-Nadal and D.~Talay.
\newblock On a {W}asserstein-type distance between solutions to stochastic differential equations.
\newblock {\em Ann.~Appl.~Probab.}, 29(3):1609--1639, 2019.

\bibitem{Bo07}
V.~I. Bogachev.
\newblock {\em Measure theory}, volume~1.
\newblock Springer Science and Business Media, 2007.

\bibitem{Bo92}
V.~I. Bogachev.
\newblock {\em Measure theory}, volume~2.
\newblock Springer Science and Business Media, 2007.

\bibitem{BoKoMe05}
V.~I. Bogachev, A.~V. Kolesnikov, and K.~V. Medvedev.
\newblock Triangular transformations of measures.
\newblock {\em Mat. Sb.}, 196(3):3--30, 2005.

\bibitem{BoLiOb23}
P.~Bonnier, C.~Liu, and H.~Oberhauser.
\newblock Adapted topologies and higher rank signatures.
\newblock {\em Ann.~Appl.~Probab.}, 33(3):2136--2175, 2023.

\bibitem{Do14}
Y.~Dolinsky.
\newblock Hedging of game options under model uncertainty in discrete time.
\newblock {\em Electron. Commun. Probab.}, 19:no. 19, 11, 2014.

\bibitem{EcPa22}
S.~Eckstein and G.~Pammer.
\newblock Computational methods for adapted optimal transport.
\newblock {\em Ann. Appl. Probab., to appear}, 2023.

\bibitem{Ed19}
M.~Eder.
\newblock Compactness in adapted weak topologies.
\newblock {\em arXiv:1905.00856}, 2019.

\bibitem{Fo22b}
H.~F\"{o}llmer.
\newblock Doob decomposition, {D}irichlet processes, and entropies on {W}iener space.
\newblock In {\em Dirichlet forms and related topics}, volume 394 of {\em Springer Proc. Math. Stat.}, pages 119--141. Springer, Singapore, [2022] \copyright 2022.

\bibitem{Fo22a}
H.~F\"{o}llmer.
\newblock Optimal couplings on {W}iener space and an extension of {T}alagrand's transport inequality.
\newblock In {\em Stochastic analysis, filtering, and stochastic optimization}, pages 147--175. Springer, Cham, [2022] \copyright 2022.

\bibitem{Gi04}
N.~Gigli.
\newblock {\em On the geometry of the space of probability measures in {$\mathbb R^n$} endowed with the quadratic optimal transport distance}.
\newblock PhD thesis, Scuola Normale Superiore di Pisa, 2004.

\bibitem{GlPfPi17}
M.~Glanzer, G.~C. Pflug, and A.~Pichler.
\newblock Incorporating statistical model error into the calculation of acceptability prices of contingent claims.
\newblock {\em Math. Program.}, 174(1-2, Ser. B):499--524, 2019.

\bibitem{He96}
M.~F. Hellwig.
\newblock Sequential decisions under uncertainty and the maximum theorem.
\newblock {\em J. Math. Econom.}, 25(4):443--464, 1996.

\bibitem{HeSc02}
M.~F. Hellwig and K.~M. Schmidt.
\newblock Discrete--time approximations of the holmstr{\"o}m--milgrom brownian--motion model of intertemporal incentive provision.
\newblock {\em Econometrica}, 70(6):2225--2264, 2002.

\bibitem{HiRoYo14}
F.~Hirsch, B.~Roynette, and M.~Yor.
\newblock Kellerer's theorem revisited.
\newblock In Springer, editor, {\em Asymptotic Laws and Methods in Stochastics. Volume in Honour of Miklos Csorgo}, Fields Institute Communications Series, 2014.

\bibitem{Ho91}
D.~Hoover.
\newblock Convergence in distribution and {S}korokhod convergence for the general theory of processes.
\newblock {\em Probab.~Theory Relat.~Fields}, 89(3):239--259, 1991.

\bibitem{Ho92}
D.~N. Hoover.
\newblock Extending probability spaces and adapted distribution.
\newblock In {\em S\'{e}minaire de {P}robabilit\'{e}s, {XXVI}}, volume 1526 of {\em Lecture Notes in Math.}, pages 560--574. Springer, Berlin, 1992.

\bibitem{HoKe84}
D.~N. Hoover and H.~J. Keisler.
\newblock Adapted probability distributions.
\newblock {\em Transactions of the American Mathematical Society}, 286(1):159--201, 1984.

\bibitem{Ka97}
O.~Kallenberg.
\newblock {\em Foundations of modern probability}.
\newblock Probability and its Applications (New York). Springer-Verlag, New York, 1997.

\bibitem{Ke72}
H.~Kellerer.
\newblock Markov-{K}omposition und eine {A}nwendung auf {M}artingale.
\newblock {\em Math. Ann.}, 198:99--122, 1972.

\bibitem{KiPfPi20}
K.~B. Kirui, G.~C. Pflug, and A.~Pichler.
\newblock New algorithms and fast implementations to approximate stochastic processes.
\newblock {\em arXiv:2012.01185}, 2020.

\bibitem{La18}
R.~{Lassalle}.
\newblock {Causal transference plans and their Monge-Kantorovich problems}.
\newblock {\em Stoch.~Anal.~Appl.}, 36(3):452--484, 2018.

\bibitem{Lo08b}
G.~Lowther.
\newblock Fitting martingales to given marginals.
\newblock {\em arXiv:0808.2319 [math]}, Aug. 2008.

\bibitem{NiSu20}
F.~Nielsen and K.~Sun.
\newblock {\em Chain Rule Optimal Transport}, pages 191--217.
\newblock Springer International Publishing, Cham, 2021.

\bibitem{Ox80}
J.~C. Oxtoby.
\newblock {\em Measure and category}, volume~2 of {\em Graduate Texts in Mathematics}.
\newblock Springer-Verlag, New York-Berlin, second edition, 1980.
\newblock A survey of the analogies between topological and measure spaces.

\bibitem{Pa22}
G.~Pammer.
\newblock A note on the adapted weak topology in discrete time.
\newblock {\em Electron. Comm. Probab., to appear}, May 2023.

\bibitem{Pf09}
G.~C. Pflug.
\newblock Version-independence and nested distributions in multistage stochastic optimization.
\newblock {\em SIAM Journal on Optimization}, 20(3):1406--1420, 2009.

\bibitem{PfPi12}
G.~C. Pflug and A.~Pichler.
\newblock A distance for multistage stochastic optimization models.
\newblock {\em SIAM J.~Optim.}, 22(1):1--23, 2012.

\bibitem{PfPi14}
G.~C. Pflug and A.~Pichler.
\newblock {\em Multistage stochastic optimization}.
\newblock Springer Series in Operations Research and Financial Engineering. Springer, Cham, 2014.

\bibitem{PfPi16}
G.~C. Pflug and A.~Pichler.
\newblock From empirical observations to tree models for stochastic optimization: convergence properties.
\newblock {\em SIAM J.~Optim.}, 26(3):1715--1740, 2016.

\bibitem{PiWe21}
A.~Pichler and M.~Weinhardt.
\newblock The nested {S}inkhorn divergence to learn the nested distance.
\newblock {\em Computational Management Science}, pages 1--25, 2021.

\bibitem{Ru85}
L.~R\"{u}schendorf.
\newblock The {W}asserstein distance and approximation theorems.
\newblock {\em Z. Wahrsch. Verw. Gebiete}, 70(1):117--129, 1985.

\bibitem{Ve70}
A.~M. Vershik.
\newblock Decreasing sequences of measurable partitions and their applications.
\newblock {\em Sov. Mat. Dokl.}, 11(4):1007 -- 1011, 1970.

\bibitem{Ve94}
A.~M. Vershik.
\newblock Theory of decreasing sequences of measurable partitions.
\newblock {\em Algebra i Analiz}, 6(4):1--68, 1994.

\bibitem{Vi03}
C.~Villani.
\newblock {\em Topics in optimal transportation}, volume~58 of {\em Graduate Studies in Mathematics}.
\newblock American Mathematical Society, Providence, RI, 2003.

\bibitem{Vi09}
C.~Villani.
\newblock {\em Optimal Transport. Old and New}, volume 338 of {\em Grundlehren der mathematischen Wissenschaften}.
\newblock Springer, 2009.

\bibitem{XuAc21}
T.~Xu and B.~Acciaio.
\newblock Conditional {COT}-{GAN} for video prediction with kernel smoothing.
\newblock In {\em NeurIPS 2022 Workshop on Robustness in Sequence Modeling}, 2022.

\end{thebibliography}



\end{document}